\documentclass[letterpaper,12pt]{amsart}
\usepackage{graphicx}
\usepackage[dvipsnames]{xcolor}
\usepackage{amsmath, amscd, amsthm, amssymb, mathrsfs}

\newcommand{\cse}{ \color{black} }

\usepackage[]{todonotes} 

\usepackage{marginnote}  
\usepackage{enumitem} 
\usepackage[colorlinks=true]{hyperref}  

\usepackage[abbrev]{amsrefs}

\newtheorem{thm}{Theorem}[section]
\newtheorem{cor}[thm]{Corollary}
\newtheorem{lem}[thm]{Lemma}
\newtheorem{prop}[thm]{Proposition}
\theoremstyle{definition}
\newtheorem{definition}[thm]{Definition}
\theoremstyle{remark}
\newtheorem{rem}[thm]{Remark}
\numberwithin{equation}{section}

\newcommand{\norm}[1]{\Vert#1\Vert}

\newcommand{\inner}[1]{\langle #1 \rangle}
\newcommand{\abs}[1]{\left\vert#1\right\vert}
\newcommand{\set}[1]{\left\{#1\right\}}
\newcommand{\R}{\mathbb R}
\newcommand{\C}{\mathbb C} 

\newcommand{\ddbar}{\sqrt{-1}\,\bar{\partial}\partial}
\newcommand{\delb}{\bar{\partial}}

\newcommand{\Lie}[1]{\mathfrak{#1}}

\newcommand{\Hilb}{\textrm{Hilb}_{k}}
\newcommand{\FS}{\textrm{FS}_{k}}
\newcommand{\Gr}{\textrm{Gr}} 
\newcommand{\D}{\mathscr{D}}

\newcommand{\Tr}{\mathrm{tr}}

\newcommand{\tr}{\mathrm{tr}}

\newcommand{\Ker}{\mathrm{Ker}}

\newcommand{\Id}{\mathrm{Id}}

\newcommand{\End}{\mathrm{End}}
\newcommand{\Hom}{\mathrm{Hom}}

\allowdisplaybreaks 

\begin{document}

\title[]{Quantization of the Laplacian operator on vector bundles,  I}%

\address{}%

\author{Julien Keller} \address{Institut de Math\'ematiques de Marseille I2M, UMR 7373, Aix - Marseille Universit\'e,
39 rue F. Joliot-Curie, 13453 Marseille Cedex 13,  France}
 \email[Julien Keller]{julien.keller@univ-amu.fr}  

\author{Julien Meyer} \address{D\'epartement de Math\'ematique, Universit\'e libre de Bruxelles, CP 218, Boulevard du Triomphe, B-1050 Bruxelles,  Belgique}
 \email[Julien Meyer]{jmeyer@ulb.ac.be}  

\author{Reza Seyyedali} \address{Department of Pure
Mathematics, University of Waterloo,
Waterloo, Ontario, N2L 3G1, Canada}
\email[Reza Seyyedali]{rseyyeda@uwaterloo.ca}



\date{May 14, 2015}
\begin{abstract}
Let $(E,h)$ be a holomorphic, Hermitian vector bundle over a polarized manifold. We provide a canonical quantization of the Laplacian operator acting on sections of the bundle of Hermitian endomorphisms of $E$. If $E$ is simple we obtain an approximation of the eigenvalues and eigenspaces of the Laplacian.
\end{abstract}
\maketitle
\tableofcontents
\pagenumbering{arabic}

\makeatletter
\renewcommand{\@todonotes@drawMarginNoteWithLine}{%
\begin{tikzpicture}[remember picture, overlay, baseline=-0.75ex]%
    \node [coordinate] (inText) {};%
\end{tikzpicture}%
\marginnote[{
    \@todonotes@drawMarginNote%
    \@todonotes@drawLineToLeftMargin%
}]{
    \@todonotes@drawMarginNote%
    \@todonotes@drawLineToRightMargin%
}%
}
\makeatother

\section{Introduction}

Consider $E$ a holomorphic vector bundle of rank $r$ over a projective manifold $X$ of complex dimension $n$ polarized by an ample line bundle $L$. Fix Hermitian metrics $h$ on $E$ and $\sigma$ on $L$ whose curvature defines a K\"ahler form $\omega$ on $X$. Given these data, one can consider the induced Bochner Laplacian over the smooth endomorphisms of $E$, denoted $$\Delta^E: C^{\infty}(X,End(E))\to C^{\infty}(X,End(E)),$$ that can be defined as follows. Let $g$ be the Riemannian metric on $TX$ associated to $\omega$ and the complex structure on $X$. Consider $\nabla^{TX}$ the Levi-Civita connection on $(TX,g)$ and $\nabla^E$ the connection on $End(E)$ induced from the Chern connection on $(E,h)$. For $\{e_i\}$ a local orthonormal frame of $(TX,g)$, one defines  $\Delta^E$ as
$$\Delta^E=-\sum_{i=1}^{2n} \left(\nabla_{e_i}^E \nabla_{e_i}^E - \nabla_{\nabla_{e_i}^{TX}e_i}^E\right).$$
Note that in the case of a line bundle or a Hermitian-Einstein metric, this reduces to the Kodaira Laplacian $\Delta_{\partial}=\partial^{*_h}_{End(E)}\partial_{End(E)}$, up to a constant factor. In general, both Laplacians are related by a Weitzenb\"ock type formula, namely 
$$\Delta_{\partial}=\frac{1}{2}\left(\Delta^E-\sqrt{-1}\Lambda_\omega F_{End(E)}\cdot \right)=\frac{1}{2}\left(\Delta^E-[\sqrt{-1}\Lambda_\omega F_{h},\cdot ]\right),$$ where $F_h\in \Omega^{1,1}(X,End(E))$ denotes the curvature endomorphism of the metric $h$. It is well known that, thanks to Akizuki-Nakano formula, $\Delta_{\partial}$ is related to $\Delta_{\bar\partial}$  which is the proper Dirac operator in our K\"ahlerian context. Remark also that  there is a decomposition $End(E)=End(E)_{SH}\oplus End(E)_{H}$ where $End(E)_{SH}$ and $End(E)_H$ denote the bundles of skew-Hermitian and Hermitian endomorphisms of $(E,h)$ respectively. It is easy to check that the Laplacian $\Delta^E$ preserves this decomposition. Throughout the paper, we restrict the action of the Bochner Laplacian to the space of sections of the bundle of Hermitian endomorphisms of $E$, which is the central object of our study.

It is well known from the work of M. Atiyah and R. Bott, that the map $$\mu_{\infty}:h\mapsto \frac{\sqrt{-1}}{2\pi} \Lambda_\omega F_h $$ can be seen as a moment map for the action of the gauge group $\mathcal{G}$ of $E$, where one has identified $Lie(\mathcal{G})$ to its dual using the volume form $\omega^n/n!$. This is a crucial fact in Donaldson's approach in the so-called Kobayashi-Hitchin-Donaldson-Uhlenbeck-Yau correspondence for vector bundles. As described by X. Wang in \cite{W2}, the  moment map $\mu_{\infty}$ can be seen as the limit of moment maps (denoted later $\bar\mu_k=\bar\mu$ and that depend on an integral parameter $k>0$) on the space of holomorphic embeddings of $M$ into Grassmannians spaces $Gr(r,N_k)$ for the $SU(N_k+1)$-action. \\
On one hand, a direct computation shows that the operator $\Delta_{\partial}$  is the gradient of the moment map $\mu_{\infty}$. On the other hand, the gradients of  the maps $\bar\mu_k$ are given by certain operators of the form $P_k^*P_k$ which are endomorphisms of $\sqrt{-1}\mathfrak{u}(N_k+1)$, the space of Hermitian matrices of size $N_k+1$. Consequently it is natural to wonder what is the relationship between $\Delta^E$, $\Delta_\partial$ and the operators $P_k^*P_k$. \\

In this paper, building on earlier works of X. Wang, J. Fine, X. Ma and G. Marinescu, we prove that the operators $P_k^*P_k$ {\it provide a quantization} of the Bochner Laplacian $\Delta^E$. At that stage it is crucial to notice that $P_k^*P_k$ lives as an operator on a finite dimensional space and is defined in a purely algebraic way. Actually, it is possible to define other approximations of the Laplacian (see for instance the $Q_k$ operator defined in \cites{D3,LM}) but they will not appear as a result of a canonical construction. \\

As a first step of this quantization, we shall see the following. Our embeddings of $M$ into Grassmannian spaces are provided by basis of $N_k+1$ holomorphic sections. For $\phi$ a  Hermitian endomorphisms with respect to $h$, one can consider the deformation by $\phi$ of the $L^2$ inner product induced by $(h,\omega^n/n!)$ on the space of holomorphic sections. Thus we obtain naturally an element $Q_{k,\phi}\in \sqrt{-1}\mathfrak{u}(N_k+1)$. Given, $\phi,\psi$ Hermitian endomorphisms, we prove that the trace of the map $$H_{I_{\mu_{\infty}}}:(\phi,\psi) \mapsto \int_M \phi \Delta_{\partial} \psi \frac{\omega^n}{n!}$$
is the quantum limit of $\tr(Q_{k,\phi} P_k^*P_k Q_{k,\psi})$, after normalization, see Theorem \ref{thm1}. The map $H_{I_{\mu_{\infty}}}$ has a natural interpretation. Firstly, note that given a moment map over a symplectic manifold $\Xi$ associated to the action of a linear reductive group $\Gamma$, one can define the integral of the moment map as a $\Gamma$-invariant functional over $\Xi\times \Gamma^\C$. It enjoys various important properties, especially that of convexity along complex orbits. Moreover its properness or coercivity can be related to the stability of the point in $\Xi$, at least in finite dimension thanks to Kempf-Ness theory. In the case of $\mu_{\infty}$ and the gauge group, the integral $I_{_{\mu_{\infty}}}$ of the moment map $\mu_{\infty}$ is known as Donaldson's functional. Its hessian is precisely $H_{I_{\mu_{\infty}}}$. This first theorem relies deeply on asymptotic results due to X. Ma and G. Marinescu.\\
In a second step we analyze the spectrum of the operators $P_k^*P_k$. If $E$ is simple we explain that the eigenvalues of $P_k^*P_k$ converge after renormalization towards the eigenvalues of the Laplacian $\Delta^E$ (Theorem \ref{thm3}). Eventually we provide a description of the eigenspaces for which a similar result holds (Theorem \ref{thm4}). For these two theorems that are key results of this paper, our method follows a strategy defined by J. Fine for the quantization of the Lichnerowicz operator \cite{Fi3}. \\

We remark that we have some flexibility in our results, i.e they are still valid if we vary the data (metric or endomorphism) in bounded sets with respect to the smooth topology. Using this fact, we derive in the last part some quantization results for sequences of balanced metrics in the sense of X. Wang when $E$ is assumed to be Mumford stable (see Theorem \ref{bal1thm}) or for balanced metrics in the sense of Zhang-Luo-Donaldson (see Theorem \ref{thmlinebal}).

\bigskip

{\small
{\bf Acknowledgements.} The first author is very grateful to Joel Fine for illuminating conversations on the subject of balanced embeddings throughout the years. Furthermore the second author wants to thank him for the numerous conversations they had and for generously sharing his ideas with him.\\

\footnotesize{\noindent The work of the first author has been carried out in the framework of the Labex Archim\`ede (ANR-11-LABX-0033) and of the A*MIDEX project (ANR-11-IDEX-0001-02), funded by the ``Investissements d'Avenir" French Government programme managed by the French National Research Agency (ANR). The first author was also partially supported by supported by the ANR project EMARKS, decision No ANR-14-CE25-0010.\\
The second author was supported by an AFR PhD grant from the Fonds National de la Recherche Luxembourg, and acknowledges travel support from the Communaut\'e fran\c caise de Belgique via an ARC and from the Belgian federal government via the PAI ``Dygest''.
}
}
\section{A quick review of the Fubini-Study geometry on Grassmannians}

Denote the space of all matrices $z \in \mathcal{M}_{r \times N}(\C)$ with rank $r$ by  $\mathcal{M}^0_{r \times N}.$ By definition,
$$\Gr(r,N)=\frac{\mathcal{M}^0_{r \times N}}{\sim},$$ where $z \sim w$ if and only if there exists $P \in GL(r, \mathbb{C})$  such that $z=Pw.$ Note that $\Gr(r,N)$ can be identified with the space of all $r$-dimensional subspaces of $\mathbb{C}^N$. The tangent bundle of $\Gr(r,N) $ is given by $$\frac{\set{(z,X)| z \in \mathcal{M}^0_{r \times N}, X\in \mathcal{M}_{r \times N}}}{\sim^{\prime}},$$ where $(z,X)\sim^{\prime} (w,Y)$ if and only if there exists $P \in GL(r, \mathbb{C})$ and $Q \in \mathcal{M}_{r \times r}(\mathbb{C})$ such that $$z=Pw, \,\,\, X=PY+Qw.$$
Under the equivalence relation $\sim^\prime$, an element of $T \Gr(r,N)$ is denoted $[z,X]$.
The Fubini-Study metric on $T \Gr(r,N)$ is given by $$\inner{[(z,X)], [(z,Y)]}_{FS}=\tr(Y^*(zz^*)^{-1}X)-\tr((zz^*)^{-1}zY^*(zz^*)^{-1}Xz^*).$$
One can easily check that this is the Fubini-Study metric induced on $\Gr(r,N)$ using the Plucker embedding. \\
Let $U_{r} \to \Gr(r,N)$ be the dual of the tautological bundle $$U_{r}^*=\{(P,v) \in \Gr(r,N) \times \mathbb{C}^N| v \in P\}.$$ The standard Hermitian inner product on $\C^N$ induces Fubini-Study metrics on $U_{r}$ and $U_{r}^*$.  There is a 1-1 correspondence between the space of linear forms on $\mathbb{C}^N$ and the holomorphic sections of $U_{r} \to \Gr(r,N):$ 
\begin{align*}
f \in (\mathbb{C}^N)^* \to &s_{f}\in H^0(\Gr(r,N),U_{r})\\
& s_{f}[z](v)=f(v), \text{ for any }[z]\in \Gr(r,N)\,\,,v\in (U_r^*)_{[z]}. 
\end{align*}
Let $e_{1},\dots e_{N} $ be the standard basis of $\C^N$ and define $s_{i}=s_{e_{i}^*}$. Then $s_{1},\dots s_{N}$ is a basis for $H^0(\Gr(r,N),U_{r})$ satisfying 
$$\sum_{i=1}^N s_{i}\otimes s_{i}^{*_{h_{FS}}}=Id_{U_{r}}$$
as an endomorphism over $Gr(r,N)$.

\begin{definition}
\label{H_A}
Let $A \in \sqrt{-1}\Lie{u}(N)$. It defines a smooth Hermitian  endomorphism (with respect to the Fubini-Study metric $h_{FS}$)  $H_A$ of $U_{r}$ as follows:
$$H_A=\sum_{i,j=1}^N A_{ji} s_{i}\otimes s_{j}^{*_{h_{FS}}}.$$
\end{definition}

\cse

\begin{definition}
Let $A \in \sqrt{-1}\Lie{u}(N)$. It induces a holomorphic vector field $\xi_A$ on $\Gr(r,N)$ as follows:
\begin{equation}\label{zeta} \xi_A(z):=[z,zA],
\end{equation}
using notations as above. Note that $\tr(H_A)$ is the associated Hamiltonian to the vector field $\xi_A$ with respect to the Fubini-Study metric. 
\end{definition}

We recall briefly the symplectic framework described in \cites{C-K,W1}.
Let us consider first $$\mu : \Gr(r,N) \rightarrow \sqrt{-1} \mathfrak{u}(N)$$ the moment map associated to the $U(N)$ action and the Fubini-Study metric $\omega_{FS}$ on $\Gr(r,N)$. Note that here we identify implicitly the Lie algebra $\sqrt{-1} \Lie{u}(N)$ with its dual using the Killing form $\inner{A,B}=\tr(AB)$. Given homogeneous unitary coordinates, the explicit formula for $\mu$ is given by
\begin{equation}
\mu([z])=z^*(zz^*)^{-1}z. \label{mu}
\end{equation}

\section{Context and Preliminaries\label{sect2}}

\subsection{Balanced metrics for bundles }
Let $(X,\omega)$ be a compact K\"ahler manifold of complex dimension
$n$ and let $(L,\sigma)$ be an ample holomorphic Hermitian line bundle over
$X$ such that its curvature satisfies 
$\frac{1}{2\pi}\ddbar\log\sigma=\omega.$ Let $E$ be a holomorphic vector
bundle of rank $r $ and degree $d$ over $X$. Since $L$ is ample we can use holomorphic sections of $E(k)$ to  embed $X$ into $\Gr(r, H^{0}(X,E(k))^*)$ for $k \gg 0$. Here $E(k)=E \otimes L^{\otimes k}$.
Indeed, for any $x \in X$, we have the evaluation map
$H^{0}(X,E(k))\rightarrow E(k)_{x}$, which sends $s$ to $s(x)$. Since
$E(k)$ is globally generated, this map is a surjection. So its dual
is an inclusion of $E(k)_{x}^* \hookrightarrow H^{0}(X,E(k))^*$, which
determines an $r$-dimensional subspace of $H^{0}(X,E(k))^*$.
Therefore we get a map $\iota: X \rightarrow \Gr(r, H^{0}(X,E(k))^*)$.
Since $E(k)$ is very ample for $k \gg 0$, $\iota$ is an embedding. Clearly we have
$\iota^*U_{r}^*=E(k)^*$, where $U_{r}^*$ is the tautological vector
bundle on $\Gr(r, H^{0}(X,E(k))^*)$, i.e. at any $r$-plane in $\Gr(r,
H^{0}(X,E(k))^*)$, the fibre of $U_{r}^*$ is exactly that $r$-plane. Any
choice of basis for $H^{0}(X,E(k))$ gives an isomorphism between
$\Gr(r, H^{0}(X,E(k))^*)$ and the standard $\Gr(r,N_{k})$, where $$N_{k}= \dim
H^{0}(X,E(k)).$$ We have the standard Fubini-Study Hermitian metric on
$U_{r}$, so we can pull it back to $E(k)$ to get a Hermitian metric. \\

Let us fix the embedding $\iota: X \hookrightarrow   \Gr(r,N)$ and identify $X$ with its image $\iota(X)$. The space of Fubini-Study type metrics on $H^0(X,i^*U_{r})=H^0(G(r,N),U_{r})$ is identified with the Bergman space. $$\mathcal{B}=\mathcal{B}_k={GL(N)}/{U(N)}\simeq \sqrt{-1} \Lie{u}(N).$$ 
Then, we can consider the integral of $\mu$ over  $X$ with respect to the volume form $$\Omega:=\displaystyle \frac{\omega^n}{n!},$$ namely
$$\bar{\mu}(\iota)=\int_X \mu(\iota(p)) \frac{\omega^n(p)}{n!} $$which induces a moment map for the $U(N)$ action over the space of all bases of $H^0(X,i^*U_{r})=H^0(G(r,N),U_{r})$. More precisely on the space $\mathfrak{M}$ of smooth maps from $X$ to $ Gr(r,H^0(\iota^* U_r)^*)$, we have a natural symplectic structure $\varpi$ defined by
$$\varpi(a,b)=\int_X \inner{a,b}\frac{\omega^n}{n!},$$ for $a,b\in T_{\iota}\mathfrak{M}$ and $\inner{.,.}$ the Fubini-Study inner product induced on the tangent vectors.
Then, $SU(N)$ acts isometrically on $\mathfrak{M}$ with the equivariant moment map $\bar{\mu}_0$ given by $$\iota\mapsto-\sqrt{-1}\left(\bar{\mu}(\iota) - \frac{\tr(\bar{\mu}(\iota))}{N}Id_{N}\right)\in \sqrt{-1}\mathfrak{su}(N).$$ Note that using a Hermitian metric $H$ on $H^0(X,\iota^*U_{r})$, one can consider an orthonormal basis with respect to $H$ and the associated embedding,  and thus it also makes sense to speak about $\bar{\mu}_0(H)$ (resp. $\bar{\mu}(H)$) up to $SU(N)$-equivariance (resp. up to $U(N)$-equivariance). 

\medskip

In the Bergman space $\mathcal{B}=GL(N)/U(N)$, we have a preferred metric associated to the volume form $\displaystyle \frac{\omega^n}{n!}$ and the moment map we have just defined, and this is precisely a balanced metric.
\begin{definition}
A metric $H\in Met(H^0(X,\iota^*U_{r}))$ (resp. the embedding $\iota$) is said to be balanced if the trace free part of $\bar{\mu}(H)$ (resp. $\bar{\mu}(\iota)$) vanishes. 
\end{definition}
In this definition, we have used the fact that for $H\in \mathcal{B}_{k}$  the quantity $\tr(\bar{\mu}_0(H))$ is well defined as function on $\mathcal{B}_k$ as it does not depend on the choice of the orthonormal basis (and thus the embedding) with respect to which it is computed, thanks to the $SU(N)$-equivariance. \\

Define the operator 
\begin{equation}
\label{P(A)}
P :\sqrt{-1} \Lie{u}(N)\rightarrow C^{\infty}(X, {T\Gr(r,N)}_{|\iota(X)})
\end{equation}
by $P(A)=\xi_A|_{X}$. Using the Fubini-Study metric on ${T\Gr(r,N)}_{|\iota(X)}$ and the volume form $\Omega$ on $X$, one obtains a $L^2$ inner product on $C^{\infty}({T\Gr(r,N)}_{|\iota(X)})$. Together with the Killing form on $\sqrt{-1} \Lie{u}(N)$
this allows us to define the adjoint map $$P^{*} : C^{\infty}({T\Gr(r,N)}_{|\iota(X)})\rightarrow \sqrt{-1} \Lie{u}(N).$$ Therefore, we have a map 
 $$P^*P: \sqrt{-1}\Lie{u}(N) \to  \sqrt{-1}\Lie{u}(N).$$

\begin{lem}

For any $A,B \in \sqrt{-1} \Lie{u}(N),$ we have $$\Tr(Bd\bar{\mu}(A))=\int_{X}\inner{\xi_{A},\xi_{B}}_{FS}\,\Omega.$$
\end{lem}

\begin{proof}
The Lemma is a direct consequence of the fact that $\bar{\mu}$ is a moment map but can also be seen by a direct computation.
Let $\{s_{1}, \dots ,s_{N}\}$ be the standard basis for $(\mathbb{C}^N)^*$ (our computation will actually be independent of the choice of the basis). We have 
$$(d\mu(A))_{ij}=\sum_{\alpha} A_{i\alpha} \inner{s_{\alpha}, s_{j}}_{FS}-\sum_{\alpha,\beta} A_{\beta \alpha}\inner{s_{\alpha}, s_{j}}_{FS}\inner{ s_{i},s_{\beta}}_{FS}.$$ Therefore,
$$\tr(Bd\mu(A))=\tr(B\mu A)-\tr(B\mu A \mu).$$ On the other hand $\mu_{ij}[z]=\inner{s_{i},s_{j}}_{FS}[z]=(z^*(zz^*)^{-1}z)_{ij}.$ Hence,
\begin{align*}\inner{\xi_{A},\xi_{B}}_{FS}&=\tr(Bz^*(zz^*)^{-1}zA)-\tr((zz^*)^{-1}zBz^*(zz^*)^{-1}zAz^*)\\&=\tr(B\mu A)-\tr(B\mu A \mu)=\tr(Bd\mu(A)).
\end{align*}
\end{proof}

 
\begin{cor} 
The map $P^*P$ is the gradient of the moment map $\bar{\mu}.$
\end{cor}

The following lemmas can be derived from the projective estimates proved in \cite{Fi1}.
\begin{lem}\label{lemm0}
For any $A,B \in \sqrt{-1} \Lie{u}(N)$ and $[z] \in \Gr(r,N),$ we have
\begin{enumerate}  \item $\Tr(H_{A}H_{B})+\inner{\xi_{A},\xi_{B}}=\tr(AB\mu).$
\item $\tr(B d\bar{\mu}(A))+\inner{H_{A},H_{B}}_{L^2}=\tr(AB\bar{\mu}).$ 
\end{enumerate}

\end{lem}

\begin{lem}\label{2normA}
We have $$\norm{d \bar{\mu}(A)}\leq 2\norm{A} \norm{\bar{\mu}}_{op},$$
where $\norm{A}$ is the Hilbert-Schmidt norm of $A.$
\end{lem}

\begin{proof}

\begin{align*} \norm{d \bar{\mu}(A)}^2&=\tr(d \bar{\mu}(A)^2)=\tr(A d\bar{\mu}(A)\bar{\mu})-\inner{H_{A},H_{d\bar{\mu}(A)}}_{L^2(X)}\\&\leq \tr(A d\bar{\mu}(A)\bar{\mu})+\norm{H_{A}}_{L^2}\norm{H_{d\bar{\mu}(A)}}_{L^2}\\&\leq \norm{A}\norm{\bar{\mu}}_{op}\norm{d\bar{\mu}(A)}+\norm{A}\norm{\bar{\mu}}_{op}\norm{d\bar{\mu}(A)}\\&\leq 2\norm{A}\norm{\bar{\mu}}_{op}\norm{d\bar{\mu}(A)}.\end{align*}
To go from the second to the third line, we used Lemma \ref{lemm0} and the fact that for Hermitian matrices $A,B,C$ one has the following inequality:
$$\abs{\tr (ABC)}\leq\Tr(A^2)^{1/2}\Tr(B^2)^{1/2}\norm{C}_{op}$$.
\end{proof}



Let $\mathcal{K}$ be the space of Hermitian metrics on $E$ and $\mathcal{B}_{k}$ be the space of Hermitian inner products on $H^{0}(X,E(k))$. Following Donaldson \cite{D3}, we can define the following maps:

\begin{itemize}

\item Define $$\Hilb:  \mathcal{K}\rightarrow \mathcal{B}_{k}$$ by

 $$\langle s,t \rangle _{\Hilb(h)}= \frac{N_{k}}{rV} \int_{X}
\langle s(x),t(x) \rangle_{h \otimes \sigma^k }\Omega(x),$$
for any $s,t \in H^{0}(X,E(k))$ where $N_{k}=\dim(H^{0}(X,E(k))) $ and $V=\textrm{  Vol}(X,\omega)$. Note that
$\Hilb$ only depends on the volume form $\Omega=\displaystyle \frac{\omega^{n}}{n!}$.

\item For the metric $H$ in $\mathcal{B}_{k}$, $\FS(H)$ is the unique metric
on $E$ such that  $$\sum s_{i}\otimes s_{i}^{*_{\FS(H) \otimes \sigma^k}}=Id_E,$$ where
$s_{1},...,s_{N_{k}}$ is an orthonormal basis for $H^{0}(X,E(k))$ with
respect to $H$. This gives the map $\FS:\mathcal{B}_{k} \rightarrow \mathcal{K}$.

\item Define a map $$\Phi_{k}: \mathcal{K} \rightarrow  \mathcal{K}$$ by $\Phi_{k}(h)=\FS \circ \Hilb(h).$
\end{itemize}

 It is not difficult to see that a balanced metric $H \in \mathcal{B}_k$ is a fixed point of the map $\Hilb \circ \FS$ map and in that case $\FS(H)$ is a fixed point of the map $\Phi_k$.

\subsection{Some asymptotic results}

Let  $(E,h)$ be a Hermitian holomorphic vector bundle on $X$. 
Let us consider $\Pi_k:L^2(E(k)) \rightarrow H^0(E(k))$ the orthogonal projection onto the space of holomorphic sections. Its integral kernel, the  (restriction to the diagonal of the) Bergman kernel,  depends on the choice of the $L^2$ inner product. When the $L^2$ inner product is obtained from the fibrewise metric $h\otimes \sigma^k$ and the volume form $\omega^n/n!$, then the Bergman kernel is denoted $B_k\in C^{\infty}(X,End(E(k)))$. More precisely, 
$$B_{k}=\sum_{i=1}^{N} s_{i}\otimes s_{i}^{*_{h\otimes \sigma^k}},$$ where $s_{1}, \dots, s_{N}$ is a basis for $H^0(E(k))$ with respect to  $\frac{Vr}{N}\Hilb(h)$.

 The following result has been proved in   \cites{Ze, Ca, Lu} and \cite{Bou} where the first term of the asymptotic expansion of the Bergman function is identified. (See also \cite[Theorem 4.1.3]{M-M} for the second term). The bundle case was studied later by X. Wang \cite{W2} using holomorphic peak sections techniques.
 \begin{thm}{\label{tian-b}}
For any Hermitian metric $h$ on $E$ and K\"ahler form $\omega \in
c_{1}(L)$,
There exist smooth endomorphisms  $A_{i}(h,\omega) \in
C^{\infty}(X,End(E))$ such that the following asymptotic expansion holds as
$k \longrightarrow \infty$,
$$B_{k}(h,\omega) \sim
k^m+A_{1}(h,\omega)k^{m-1}+\dots.$$
In particular
$$A_{1}(h,\omega)= \frac{\sqrt{-1}}{2\pi} \Lambda_\omega F_{(E,h)}+\frac{1}{2}
S(\omega) Id_{E},$$ where $\Lambda_\omega $ is the trace operator acting
on $(1,1)$-forms with respect to the K\"ahler form $\omega$,
$F_{(E,h)}$ is the curvature of $(E,h)$ and $S(\omega)$ is the
scalar curvature of $\omega$.

Moreover, the asymptotic expansion holds in $C^{\infty}$. More
precisely, for any positive integers $a$ and $p$, there exists a
positive constant $C_{a,p,\omega,h}$ such that
$$ \norm{B_{k}(h,\omega)-\big(k^m+\dots+A_{p}(h,\omega)k^{m-p} \big)}_{C^a}\leq C_{a,p,\omega,h} k^{m-p-1}.$$ Moreover the expansion is uniform in
the sense that there exists a positive integer $s$ such that if
$h$ and $\omega$ run in a bounded family in $C^s$ topology and
$\omega$ is bounded from below, then the constants $C_{a,p,\omega,h}$
are bounded by a constant depending only on $a$ and $p$.

\end{thm}

Let $f,g\in C^{\infty}(X,End(E))$. We define $$T_{k,f}=\Pi_k\circ f \circ \Pi_k :  L^2(X,E(k))\to L^2(X,E(k)),$$ where $$\Pi_k: L^2(X,E(k)) \to H^0(E (k)) $$ is the orthogonal projection with respect to $\frac{Vr}{N}\Hilb(h)$. We also define $$T_{k,f,g}=T_{k,f} \circ T_{k,g} :L^2(X,E(k))\to L^2(X,E(k)).$$
Let $s_{1}, \dots s_{N}$ be an orthonormal basis  for $H^0(E(k))$ with respect to $\frac{Vr}{N}\Hilb(h)$.  Let $s \in L^2(X,E(k))$ and $y \in X$, then 
\begin{align*}(T_{k,f}s)(y)&=\big(  \Pi_k\circ f \circ \Pi_k \big)(s) (y) = \Pi_k\circ f \Big(\sum_{i=1}^{N} \inner{s,s_{i}}_{L^2} s_{i} \Big)(y)\\&=\sum_{i,j=1}^{N}\int_{X \times X} \inner{fs_{i}, s_{j}}(z)\inner{s,s_{i}}(x)\Omega(x)\Omega(z) s_{j}(y)\\&=\int_{X} K_{k,f}(x,y)s(x)\Omega(x),\end{align*} 

where, $$K_{k,f}(x,y)=\int_{X} \sum_{i,j=1}^{N} \inner{fs_{i}, s_{j}}(z)s_{j}(y)\otimes s_{i}^*(x) \Omega(z).  $$ This shows that $ K_{k,f}(x,y)$ is the integral kernel for the operator $T_{k,f}$, i.e. for any $y\in X$ and $s \in L^2(X,E(k))$, we have $$(T_{k,f}s)(y)=\int_{X} K_{k,f}(x,y)s(x)\Omega(x).$$
We also have, \begin{align*}  (T_{k,f,g}s)(y)&=\int_{X} K_{k,f}(x,y)(T_{k,g}s)(x)\Omega(x)\\&=\int_{X} \Big(\int_{X} K_{k,f}(x,y)K_{k,g}(z,x)s(z)\Omega(z)\Big)\Omega(x)\\&=\int_{X} K_{k,f,g}(x,y)s(x)\Omega(x),\end{align*}
where,

$$K_{k,f,g}(x,y)=\int_{X} K_{k,f}(z,y) \circ K_{k,g}(x,z)\Omega(z).$$ 
This shows that $ K_{k,f,g}(x,y)$ is the integral kernel for the operator $T_{k,f,g}$, i.e. for any $y\in X$ and $s \in L^2(X,E(k))$, we have  $$(T_{k,f,g}s)(y)=\int_{X} K_{k,f,g}(x,y)s(x)\Omega(x).$$ 

We refer to \cites{M-M2,Fi3} for the next proposition.
\begin{prop}{\label{toeplitzexp}}
Let $f,g\in C^{\infty}(X,End(E))$. The restriction of the kernels $K_{k,f}$ and $K_{k,f,g}$ to the diagonal have the following asymptotic expansions as $k\rightarrow + \infty :$
$$
K_{k,f}=b_{f,0}k^n+b_{f,1}k^{n-1}+b_{f,2}k^{n-2}+O(k^{n-3})
$$
where
\begin{eqnarray*}
b_{f,0}&=&f,\\
b_{f,1}&=&\frac{S(\omega)}{2 }f + \frac{\sqrt{-1}}{4\pi} (\Lambda_{\omega}F_{h_E} f + f \Lambda_{\omega}F_{h_E}) - \frac{1}{4\pi} \Delta^E f.\\
\end{eqnarray*}
In the case when $E$ is a line bundle, we have that
$$
b_{f,2}=A_2(\omega)f+\frac{1}{32\pi^2}\Delta^2f-\frac{1}{8\pi^2}S(\omega)\Delta f
	+\frac{1}{8\pi^2}(Ric,\ddbar f).
$$
Moreover, $$K_{k,f,g}=b_{0,f,g}k^n + b_{1,f,g}k^{n-1}+O(k^{n-2})$$ where
\begin{eqnarray*}
b_{0,f,g}&=&fg,\\
b_{1,f,g}&=& \frac{1}{2} S(\omega)fg  +\frac{\sqrt{-1}}{4\pi} (\Lambda_{\omega}F_{h_E} fg + fg\Lambda_{\omega} F_{h_E}) \\
&&- \frac{1}{4\pi}(f \Delta^E g +(\Delta^E f)g )+ \frac{1}{2\pi} \langle \bar{\partial}^E f, \nabla^{1,0} g\rangle_\omega.
\end{eqnarray*}
Moreover these expansions are uniform in the endomorphisms $f,g$ if $f,g$ vary in a subset of $C^{\infty}(X,End(E))$ which is bounded for the $C^{\infty}$-topology. Eventually the expansions are uniform when the metric $h$ on $E$ varies in a set of uniformly equivalent metrics lying in a bounded set for the $C^{\infty}$-topology.
\end{prop}

\subsection{The Lichnerowicz operator}
Later, we shall need to introduce the  Lichnerowicz operator.
Let $L\rightarrow X$ be an ample holomorphic line bundle over a compact K\"ahler manifold $X$. Fix a Hermitian metric $h$ on $L$ such that its curvature is a K\"ahler form $\omega$. Let $\D: C^\infty(X,\R)\rightarrow\Omega^{0,1}(TX)$ be the operator defined by
$$
	\D(f)=\bar\partial(v_f)
$$
where $v_f$ is the Hamiltonian vector field corresponding to $f$ via $\omega$. Intuitively, $\D(f)$ measures the failure of the Hamiltonian vector field $v_f$ of being holomorphic. Write $\D^*$ for its $L^2$-adjoint. Furthermore, denote by $D$ the linearization of the scalar curvature. More precisely, if $h_t$ is a path of Hermitian metrics on $L$ given by $h_t=e^{ft}$, then we define
$$
	D(f)=\frac{\partial S(h_t)}{\partial t}.
$$
\begin{lem}
\label{linerarization_scalar_curv}
$$
	D(f)=\frac{1}{4\pi}\left(\Delta^2f-2(Ric,2\ddbar f)\right)
$$
and
$$
	\D^*\D(f)=2\pi\left(D(f)+(dS,df)\right)
$$
\end{lem}
See for example \cite{D4} and \cite{Fi3} as a reference.

\section{ The Gradient of the moment map $\bar{\mu}$ } \label{Sect1}

In this section, we fix an Hermitian metric $h$ on $E$. It induces an Hermitian inner products $\Hilb(h)$ on $H^0(X,E(k))$ for each $k$. Choose any orthonormal basis of $H^{0}(X,E(k))$ to identify $ \Gr(r, H^{0}(X,E(k))^*)$ with the standard Grassmannian $\Gr(r,N_{k})$. This gives us a sequence of embeddings $\iota_{k}: X \rightarrow \Gr(r, N_{k})$. Let us recall the construction of the operator $P_k^*P_k$ in this setting.
Any $A \in \sqrt{-1}\Lie{u}(N_{k})$ defines a holomorphic vector field $\zeta_A$ on $\Gr(r,N_{k})$. As before, we define $$P_k : \sqrt{-1} \Lie{u}(N)\rightarrow C^{\infty}(X, {TGr(r,N_{k})}_{|\iota(X)})$$ by
$P_k(A)=P(A)=\xi_A|_{X}$. Using the Killing form on $\sqrt{-1}\Lie{u}(N_k)$ and the $L^2$ inner product on $C^{\infty}(X, {TGr(r,N_{k})}_{|\iota(X)})$ induced by the Fubini-Study metric on the tangent space  we get an adjoint map $$P^{*}_k : C^{\infty}({TGr(r,N_{k})}_{|\iota_{k}(X)})\rightarrow \sqrt{-1} \Lie{u}(N_{k}).$$  Thus, we have a sequence of maps $P^*_k P_k : \sqrt{-1} \Lie{u}(N) \to  \sqrt{-1} \Lie{u}(N).$






\begin{definition}

Let $h$ be a metric on $E$ and $\phi \in C^{\infty}(X,End(E))$ Hermitian with respect to $h$. Define $$Q_{\phi,k}=\frac{d}{dt}|_{t=0} \Hilb(h(I+t\phi)).$$

\end{definition}

\begin{lem}
Let $s_{1}, \dots s_{N}$ be an orthonormal basis  for $H^0(E(k))$ with respect to $\frac{Vr}{N}\Hilb(h)$, then $$\big(Q_{\phi,k} \big)_{ij}=\int_{X} \inner{s_{i},\phi s_{j}}_{h \otimes \sigma^k}\Omega.$$

\end{lem}


\begin{thm}\label{thm1} Let $h$ be a metric on $E$ and $\phi \in C^{\infty}(X,End(E))$ Hermitian with respect to $h$. 
Then we have the asymptotics
$$
	\tr\left( Q_{\phi,k} P_k^*P_k Q_{\phi,k}\right)=a_1k^{-1}+a_2 k^{-2}+O(k^{-3})  
$$
where the leading order term $a_1$ is given by
\begin{align*}
a_1&= \frac{1}{4\pi}  \int_{X}\tr\big( \phi\Delta^E \phi \big)\, \Omega,\\
&= \frac{1}{2\pi}  \int_{X}\tr\big( \phi\Delta_\partial \phi \big)\, \Omega,\\
&= \frac{1}{2\pi}  \int_{X}\tr\big( \phi\Delta_{\bar\partial} \phi \big) \, \Omega.
\end{align*}
Moreover, in the case where $E$ is a line bundle, the second order coefficient is given for all $\phi\in C^\infty(X,\R)$ by 
$$
	a_2=\frac{1}{16\pi^2}\int_X \phi \D^*\D \phi -2\phi\Delta^2 \phi\,  \Omega.
$$
These estimates are uniform in the endomorphism $\phi$ if $\phi$ varies in a subset of $C^{\infty}(X,End(E)_H)$ which is bounded for the $C^{\infty}$-topology. The estimate is uniform when the metric $h$ on $E$ varies in a set of uniformly equivalent metrics lying in a compact set for the $C^{\infty}$-topology.
\end{thm}
\begin{proof}

For any $A,B \in \sqrt{-1}\Lie{u}(N)$, we have $$\tr(AP_k^*P_k(B))=\inner{P_k(A),P_k(B)}=\int_{X}\inner{\xi_{A},\xi_{B}}_{h_{FS}} \Omega .$$
Therefore,  using Lemma \ref{lemm0},
\begin{align}
\label{formula_PP}
\tr\left( Q_{\phi,k} P_k^*P_k Q_{\phi,k}\right)
&=\int_{X}\inner{\xi_{Q_{\phi,k}},\xi_{Q_{\phi,k}}} \Omega\nonumber \\
&=\int_{X}\tr(Q_{\phi,k}^{ 2 }\mu) \Omega -\int_{X}\tr(H_{Q_{\phi,k}}^2)\Omega .
\end{align}
Let $s_{1},\dots ,s_{N}$ be an orthonormal basis with respect to $\frac{Vr}{N}\cse\Hilb(h)$. We have 
\begin{align*} \int_{X}\tr(&Q_{\phi,k}^{ 2 }\mu) \Omega \\
=& \tr(Q_{\phi,k}^{ 2 }\int_{X}\mu  \Omega  )\\
=&\sum_{i,j,l} \int_{X}\inner{s_{i},\phi s_{j}} \Omega \int_{X}\inner{s_{j},\phi s_{l}} \Omega \int_{X}\inner{s_{l},B_{k}^{-1}s_{i}} \Omega\\
=& \Tr\int_{\hspace{-0.05cm}X\hspace{-0.05cm}\times\hspace{-0.05cm} X\hspace{-0.05cm}\times\hspace{-0.05cm} X} { \sum_{i,j,l} \inner{s_{i},\phi s_{j}}(x)\inner{s_{j},\phi s_{l}}(y) s_{l} \otimes s_{i}(z)\hspace{-0.1cm} \circ \hspace{-0.1cm}B_{k}^{-1}(z)\Omega(x)\Omega(y)\Omega(z)} \\ 
=&  \tr \int_{X} K_{\phi, \phi,k}(x)(B_{k}^{-1}(x)) \Omega(x) \cse
\end{align*}
Using Theorem \ref{tian-b} and Proposition \ref{toeplitzexp}, one can compute the first two terms in the asymptotic expansion of this expression. After a short computation, one gets
\begin{align*}
\int_{X}&\tr ( \phi^2) \Omega \cse+k^{-1}\tr \Big(\int_{X}\frac{1}{2\pi} \inner{\bar\partial \phi, \nabla^{1,0} \phi}_{\omega}  \Omega  -\int_{X}\frac{1}{2\pi}\phi \Delta^E \phi\Big) \Omega+O(k^{-2})\\
=& \int_{X}\tr(\phi^2) \Omega \cse+\frac{k^{-1}}{2\pi}\int_{X}\big( \abs{\bar\partial \phi}^2-tr(\phi \Delta^E \phi ) \big) \Omega \cse+O(k^{-2}).\end{align*}
On the other hand,
\begin{align*}H_{Q_{\phi,k}}&=\sum_{i,j} \int_{X} \inner{s_{i}, \phi\cse s_{j}} s_{j}\otimes s_{i}^{*_{\FS(h)}} \Omega \cse=\sum_{i,j} \int_{X} \inner{s_{i}, \phi\cse s_{j}} s_{j}\otimes s_{i}^{*_{B_{k}^{-1}h}}\Omega \\&=\sum_{i,j} \int_{X} \inner{s_{i}, \phi\cse s_{j}} s_{j}\otimes s_{i}^{*_{h}} \circ B_{k}^{-1} \Omega \cse=K_{k,\phi} \circ B_{k}^{-1}.\end{align*} 
Using asymptotic expansions for $B_{k}$ and $K_{\phi,k}$, we have 
\begin{align}H_{Q_{\phi,k}}=&\left(\phi +k^{-1}\Big(\frac{S(\omega)}{2 }\phi + \frac{\sqrt{-1}}{4\pi} (\Lambda F_{h}\phi + \phi \Lambda F_{h}  ) - \frac{1}{4\pi} \Delta^E \phi \Big)+\dots\right)\nonumber \\&\times \left(I-k^{-1}\Big(\frac{S(\omega)}{2 } + \frac{\sqrt{-1}}{2\pi} \Lambda F_{h}\Big)+\dots \right)\nonumber \\=&\phi+k^{-1}\left(  \frac{\sqrt{-1}}{4\pi} (\Lambda F_{h}\phi + \phi \Lambda F_{h}-2 \phi \Lambda F_{h}  ) - \frac{1}{4\pi} \Delta^E \phi  \right)\nonumber \\
&+O(k^{-2}).\label{Hphik}\end{align}
Therefore,

\begin{align*}\tr(H_{Q_{\phi,k}}^2)&=tr(\phi^2)-\frac{k^{-1}}{2\pi}\Tr(\phi \Delta^E \phi )+O(k^{-2}).\end{align*} 
Thus, \begin{align*}\tr\left( Q_{\phi,k} P_k^*P_k Q_{\phi,k}\right)&=\int_{X}\tr(Q_{\phi,k}\mu) \Omega \cse-\int_{X}\tr(H_{Q_{\phi,k}}^2) \Omega \cse\\&=\frac{k^{-1}}{2\pi}\int_X \vert \bar\partial \phi\vert^2 \Omega+O(k^{-2}). \end{align*} On the other hand, the Weitzenb\"ock formula implies that \begin{equation}\Delta_{\bar{\partial}}\phi=\frac{1}{2}\left(\Delta^E\phi+\sqrt{-1}[\Lambda F_{h},\phi]\right).\end{equation}
Hence, \begin{align*} \frac{1}{2\pi}\int_X \vert \bar\partial \phi\vert^2\Omega &=\frac{1}{2\pi}\int_X \Tr(\phi \Delta_{\bar{\partial}}\phi)\Omega\\
&= \frac{1}{4\pi}\int_X \Tr(\phi \Delta^E\phi)\Omega +\frac{\sqrt{-1}}{2\pi}\int_X \Tr(\phi [\Lambda F_{h},\phi])\Omega\\&=\frac{1}{4\pi}\int_X \Tr(\phi \Delta^E\phi)\Omega\end{align*}
Thus, using the Akizuki-Nakano identity relating $\Delta_{\bar\partial}$ to $\Delta_{\partial}$ gives us the leading order coefficient. \\
Now, we explain how compute the second order coefficient in the case $E$ is a line bundle. The only thing we have to do is to go $1$ term further is the asymptotic expansion of the two terms in equation \ref{formula_PP}.\\
Let us start with the first term. We already know that 
$$
	\int_{X}\tr(Q_{\phi,k}^{ 2 }\mu) \Omega 
	= \int_{X} K_{\phi, \phi,k}(x)(B_{k}^{-1}(x)) \Omega(x)
$$
Using Theorem \ref{tian-b} and Taylor expansion, one easily sees that 
$$
	B_{k}^{-1}=k^{-n}\left(1-\frac{S}{2}k^{-1}+\left(\frac{S^2}{4}-A_2\right)k^{-2}+O(k^{-3})\right).
$$
Moreover, using the asymptotic expansion for $K_{\phi,\phi,k}$ given in Proposition \ref{toeplitzexp}, one gets
\begin{align*}
K_{\phi, \phi,k}B_{k}^{-1}=&\left(1-\frac{S}{2}k^{-1}+\left(\frac{S^2}{4}-A_2\right)k^{-2}+O(k^{-3})\right) \\
&\times\Big{\{}
\phi^2+\left(\frac{S}{2}\phi^2-\frac{1}{2\pi} \phi\Delta\phi+\frac{1}{4\pi}\abs{df}^2\right)k^{-1}  \\
&\hspace{4cm} +b_{\phi,\phi,2}k^{-2}+O(k^{-3})
\Big{\}}.
\end{align*}
Since we already computed the $k^0$ and $k^{-1}$-terms above, we only focus on the $k^{-2 }$-term. An easy computation shows that the coefficient of the $k^{-2}$ term is given by
$$
b_{\phi,\phi,2}+\frac{S}{4\pi}\phi\Delta \phi-\frac{S}{8\pi}\abs{d\phi}^2-A_2\phi^2,
$$
and hence the $k^{-2}$-coefficient of $\int_{X}\tr(Q_{\phi,k}^{ 2 }\mu) \Omega $ is
\begin{equation}
\label{equa_1}
\int_X \left(b_{\phi,\phi,2}+\frac{S}{4\pi}\phi\Delta \phi-\frac{S}{8\pi}\abs{d\phi}^2-A_2\phi^2 \right)\Omega.
\end{equation}
Now consider the second term of the right-hand side of equation \ref{formula_PP}. We need the asymptotic expansion of $H_{Q_{k,\phi}}$ up to the $k^{-2}$ term. We have
\begin{align*}
H_{Q_{k,\phi}}
&=K_{k,\phi}B_b^{-1} \\
&=\left(1-\frac{S}{2}k^{-1}+\left(\frac{S^2}{4}-A_2\right)k^{-2}+O(k^{-3})\right) \\
&\hspace{1.5cm}\times \left(\phi+\left(\frac{S}{2}\phi-\frac{\Delta\phi}{4\pi}\right)k^{-1}+b_{\phi,2}k^{-2}+O(k^{-3})\right).
\end{align*}
Expanding this expression, we get
$$
\phi-\frac{\Delta\phi}{4\pi}k^{-1}+\left(b_{\phi,2}+\frac{S\Delta\phi}{8\pi}-A_2\phi\right)k^{-2}+O(k^{-3}).
$$
A easy calculation shows than that the $k^{-2}$-coefficient of $H_{Q_{k,\phi}}^2$ is given by
\begin{align*}
2\phi b_{\phi,2}+\frac{S\phi\Delta\phi}{4\pi}-2A_2\phi^2+\frac{(\Delta\phi)^2}{16\pi^2}
\end{align*}
Hence the $k^{-2}$-coefficient of $\int_X H_{Q_{k,\phi}}^2 \Omega $ is
\begin{equation}
\label{equa_2}
\int_X\left(2\phi b_{\phi,2}+\frac{S\phi\Delta\phi}{4\pi}-2A_2\phi^2+\frac{(\Delta\phi)^2}{16\pi^2}\right)\Omega.
\end{equation}
Putting equations \ref{equa_1} and \ref{equa_2} together we get the $k^{-2}$-coefficient of $\tr(Q_{k,\phi}P_k^*P_k Q_k\phi)$:
$$
\int_X \left(
b_{\phi,\phi,2}-\frac{S\abs{d\phi}^2}{8\pi}-2\phi b_{\phi,2}+A_2\phi^2-\frac{\phi\Delta^2\phi}{16\pi^2}
\right)\Omega.
$$
Now we use the fact that  $\int_X K_{\phi,\psi,k}=\int_X \phi K_{\psi,k}$ which implies that $\int_X b_{\phi,\psi,k}=\int_X \phi b_{\psi,k}$. Using the formula for $b_{\phi,2}$ given is Proposition \ref{toeplitzexp} we get 
$$
\int_X
\left(
-\frac{3\phi\Delta^2\phi}{32\pi^2}+\frac{S}{8\pi}\left(\phi\Delta \phi-\abs{d\phi}^2\right)-\frac{\phi}{8\pi^2}\left(Ric,\ddbar\phi\right)
\right)
\Omega.
$$
We will simplify this by using the following identities which can be proven using Leibniz's rule and integration by parts:
$$
	\int_X \phi(dS,d\phi)\, \Omega = \frac{1}{2}\int_X \phi^2\Delta S\, \Omega = \int_X S(\phi\Delta \phi -|d\phi|^2)\,\Omega.
$$
We get
$$
\int_X
\frac{\phi}{8\pi}\left(
(dS,d\phi)-\frac{3\Delta^2\phi}{4\pi}-\frac{2(Ric,2\ddbar\phi)}{4\pi}
\right)\Omega.
$$
Using the formulas from Lemma \ref{linerarization_scalar_curv} for the Lichnerowicz operator, this can be written as
$$
\frac{1}{16\pi^2}\int_X \left(\phi\D^*\D\phi-2\phi\Delta^2\phi\right)\Omega,
$$
which concludes the proof.
\end{proof}
By symmetry, we get the following corollary.
\begin{cor}
 Let $h$ be a metric on $E$ and $\phi,\psi \in C^{\infty}(X,End(E))$ Hermitian with respect to $h$. 
Then we have the asymptotics 
 $$\tr\left( Q_{\phi,k} P_k^*P_k Q_{\psi,k}\right)=\frac{1}{8\pi k}  \int_{X}\tr\big( \phi\Delta^E \psi + \psi \Delta^E \phi \big)\Omega+O(k^{-2}).$$  
 Moreover this estimate is uniform in the endomorphisms $\phi,\psi$ if $\phi,\psi$ vary in a subset of $C^{\infty}(X,End(E))$ which is bounded for the $C^{\infty}$-topology. The estimate is uniform when the metric $h$ on $E$ varies in a set of uniformly equivalent metrics lying in a compact set for the $C^{\infty}$-topology. 
\end{cor}

\begin{rem}
 In all the section we have set $\omega=c_1(\sigma)$ and have assumed that $\omega^n/n!=\Omega$ pointwisely. Nevertheless, it is possible to define the operators $\Hilb$, $Q_{k,.}$ and $P_k$ with respect to a volume form $\Omega$ such that $\Omega\neq \omega^n/n!$. In that case, Theorem \ref{thm1} still holds.
\end{rem}

\section{Eigenvalues and Eigenspaces} \label{Sect2}

For $j\geq 0$, let $\lambda_j$ be the eigenvalues of the Bochner Laplacian $\Delta=\Delta^E$ acting on the space of smooth sections of the bundle $End(E)_H$ of Hermitian endomorphisms of $E$. We use the convention that $0\leq \lambda_{j}\leq \lambda_{j+1}.$ If we set $E_r$ the space generated by the eigenspaces $$\{v \in C^{\infty}(End(E)_H)|(\Delta^E-\lambda_j Id)v=0\}$$ for $0\leq j\leq r$, then $$\lambda_{r+1}=\min_{\phi \in E_r^\perp} \frac{\Vert \nabla \phi\Vert^2_{L^2} }{\Vert \phi\Vert^2_{L^2}}.$$ Note that $\dim E_{r} \geq r+1$ and the equality holds if and only if $\lambda_{r+1} > \lambda_r$. Let $\nu_{0,k}\leq ...\leq \nu_{M_{k}, k}$ the eigenvalues of the operator $ P_k^*P_k,$ where $M_{k}+1=\dim \Lie{u}(N_{k}).$  Define $F_{r,k}$ to be the space generated by the eigenspaces $$\{A \in \sqrt{-1}\Lie{u}(N_{k})|(P_k^*P_k-\nu_{j,k}Id)A=0\}$$ for $ 0\leq j \leq r$. Then 
$$\nu_{r+1,k}=\min_{B\in F_{r,k}^{\perp}}\frac{\Vert P_k B \Vert^2}{\Vert B\Vert^ 2}.$$
Note that $\dim F_{r,k} \geq r+1$ and the equality holds off $\nu_{r+1,k} >\nu_{r,k}$.
We write write $F_{p,q,k}\subset \sqrt{-1} \Lie{u}(N_{k})$ for the span of $\nu_{j,k}$-eigenspaces of $P_k^*P_k$ with $p\leq j \leq q$. 
{\bf From now we will assume that $E$ is simple.}
\begin{thm}\label{thm3} Under the setting as above and assuming that $E$ is a simple vector bundle, 
 for each $j\geq 0$, one has 
$$\nu_{j,k}=\frac{\lambda_j}{4 \pi k^{n+1}}+O(k^{-n-2}),$$
when $k\rightarrow +\infty$.
\end{thm}

\begin{thm}\label{thm4} Under the setting as above assume that $E$ is a simple vector bundle.
 Fix an integer $r>0$. There is a constant $C>0$ such that for all $A,B\in F_{r,k}$,
$$\Big\vert \Tr(AB)-k^{n} \langle H_A,H_B\rangle_{L^2} \Big\vert \leq Ck^{-1}\Tr(A^2)^{1/2}\Tr(B^2)^{1/2}.$$
Moreover, let us fix integers $0<p<q$ such that $$\lambda_{p-1}<\lambda_p = \lambda_{p+1} = ... = \lambda_q <\lambda_{q+1}.$$  Given an eigenvector $\phi \in \Ker(\Delta^E- \lambda_p Id)$, let $A_{\phi,k}$ denote the point $F_{p,q,k}$ with $H_{A_{\phi,k}}$ nearest to $\phi$ as measured in $L^2$. Then 
$$\Vert H_{A_{\phi,k}}-\phi\Vert^2_{L^2}= O(k^{-1}),$$
and this estimate is uniform in $\phi$ if we require that $\Vert \phi\Vert_{L^2}=1$.
\end{thm}
In both theorems, the estimates are uniform when the metric varies in a family of uniformly equivalent metrics
which is compact for the smooth topology.

\smallskip

\smallskip
\begin{center}
\textbf{Induction hypotheses (I).}

\rule{5cm}{0.05cm}
\end{center}
\smallskip

Let $r$ be a non-negative integer. We call the following statement the $r^{\text{th}}$  inductive hypotheses.
\begin{enumerate}
 \item For each $j=0,...,r$, 
$$\nu_{j,k}=\frac{\lambda_j}{ 4 \pi k^{n+1}}+O(k^{-n-2})$$
\item There exists a constant $C>0$ such that for all $A,B\in F_{r,k}$, 
$$\Big\vert \Tr(AB)-k^n \langle H_A,H_B\rangle_{L^2} \Big\vert \leq ck^{-1}\Tr(A^2)^{1/2}\Tr(B^2)^{1/2}$$
\item Fix integers $0<p<q\leq r$, such that  $\lambda_{p-1}<\lambda_p = \lambda_{p+1} = ... = \lambda_q <\lambda_{q+1}$. Given $\phi \in \Ker(\Delta^E- \lambda_p Id)$, let $A_{\phi,k}$ denote the point in $F_{p,q,k}$ with $H_{A_{\phi,k}}$ nearest to $\phi$ as measured in $L^2$. Then 
$$\Vert H_{A_{\phi,k}}-\phi\Vert^2_{L^2}= O(k^{-1})$$
and this estimate is uniform in $\phi$ if we require that $\Vert \phi\Vert_{L^2}=1$. 
\end{enumerate}
Let us assume that the eigenvalues satisfy $\lambda_r<\lambda_{r+1}=... = \lambda_s<\lambda_{s+1}$. To carry out the induction we will prove that the $r^{\text{th}}$ inductive hypotheses imply the $s^{\text{th}}$ inductive hypotheses. This will prove simultaneously both Theorems \ref{thm3} and \ref{thm4} using the induction scheme (I).
\begin{center}
\rule{5cm}{0.05cm}
\end{center}
\smallskip
\subsection{Level 0 of the induction (I)}
The first eigenvalue of $\Delta^E$ is $\lambda_0=0$ and $\Ker(\Delta)=\mathbb{R}Id_{E}$ since $E$ is simple. Therefore, $\lambda_1>\lambda_0$.
Since the first eigenvalue $\nu_{0,k}$ of $P_k^* P_k$ is also $0$, step 1  of the induction obviously holds. 
Furthermore, $Id=Id_{N}$ spans the $\nu_{0,k}$-eigenspace, so the step 3 holds immediately too. With $\tr(Id^2)=N_{k}=rVk^n+O(k^{n-1})$ and $H_{Id}=Id_{E}$, one gets easily Step 2 since $\int_X \inner{H_{Id},H_{Id}}\Omega =rV.$ Hence, the induction process is valid at the base level.

\subsection{Upper bound on the eigenvalues}
We start by giving an asymptotic upper bound of the eigenvalues of the operator $P_k^*P_k$. 
\begin{definition}
Define the projection $$\pi_{k}: L^2(X,End(E)) \to F_{r,k},$$ by $\pi_{k}(\phi)$ is the $L^2$-orthogonal projection of $Q_{\phi,k}$ onto $F_{r,k}.$

\end{definition}

\begin{lem}\label{easybound}
 Assume that $\lambda_r<\lambda_{r+1}= ... = \lambda_s <\lambda_{s+1}$ and that the inductive hypothesis at level $r$ holds. Then for all $j=r+1, ... ,s$ one has 
$$\nu_{r+1,k}\leq \frac{\lambda_{r+1}}{ 4\pi  k^{n+1}}+O(k^{-n-2})$$
\end{lem}
\begin{proof}
Define $I=\{k| \dim F_{r,k}=r+1\}.$ Hence, $k \in I$ if and only if  $\nu_{r,k} < \nu_{r+1,k}$. Fix $k \in I$, we have $\Ker(\pi_{k}|_{E_{r+1}}) \neq \emptyset,$ since $\dim F_{r,k} <\dim E_{r+1}.$ Let $\phi \in \Ker(\pi_{k})|_{E_{r+1}}$ such that $\norm{\phi}_{L^2}=1.$ Therefore, 
$$\nu_{r+1,k} \leq \frac{\norm{P_{k}Q_{\phi,k}}^2}{\Tr(Q_{\phi,k}^2)}=\frac{\Tr(Q_{\phi,k}P_{k}^*P_{k}Q_{\phi,k})}{\Tr(Q_{\phi,k}^2)},$$ since $Q_{\phi,k} \bot F_{r,k}.$ On the other hand, 
$$\Tr(Q_{\phi,k}P_{k}^*P_{k}Q_{\phi,k})=\frac{1}{ 4 \pi k}\int \Tr(\phi\Delta^E \phi )+O(k^{-2})\leq\frac{\lambda_{r+1}}{ 4 \pi k} \norm{\phi}_{L^2}^2+O(k^{-2}),$$
$$\Tr(Q_{\phi,k}^2)=k^n\int \Tr(\phi^2)+O(k^{n-1})=k^n+O(k^{n-1}).$$ Therefore, for any $k \in I,$ we have 
$$\nu_{r+1,k} \leq \frac{\lambda_{r+1}}{ 4 \pi k^{n+1}}+O(k^{-n-1}). $$ On the other hand, the induction hypotheses implies that for $k \not \in I,$
$$\nu_{r+1,k} = \nu_{r,k} \leq \frac{\lambda_{r}}{ 4 \pi k^{n+1}}+O(k^{-n-1})\leq\frac{\lambda_{r+1}}{ 4 \pi k^{n+1}}+O(k^{-n-1}).$$

\end{proof}


\subsection{Some estimates}

We need the following facts, see \cite[Lemmas 20 \&  21]{Fi3}, see also \cite{Fi1}.

\begin{lem}\label{0}
Let $h$ be a Hermitian metric on $E.$
Then $$\norm{\bar{\mu}(\Hilb(h))-\frac{rV}{N_{k}}Id}_{op}=O(k^{-n-1}) \,\, \textrm{as } \, k\to +\infty. $$ Moreover, the convergence is uniform if the metric $h$ runs in a bounded family.

\end{lem}

\begin{proof}

Let $\{s_{1}, \dots ,s_{N_k}\}$ be an orthonormal basis for $H^{0}(X,E(k))$ with respect to $H_{k}=\Hilb(h)$, i.e $$\int_{X}\inner{s_{i},s_{j}}_{h\otimes \sigma^k} \frac{\omega^n}{n!}=\frac{rV\delta_{ij}}{N_{k}}.$$ We have
\begin{align*}\bar{\mu}_{ij}&=\int_{X}\inner{s_{i},B^{-1}_{k}(h)s_{j}}_{h\otimes \sigma^k} \frac{\omega^n}{n!}\\
&=\int_{X}\inner{s_{i},(Id_E+\epsilon_{k})s_{j}}_{h\otimes \sigma^k} \frac{\omega^n}{n!}\\
&=\frac{rV\delta_{ij}}{N_{k}}+\int_{X}\inner{s_{i},\epsilon_{k}s_{j}}_{h\otimes \sigma^k} \frac{\omega^n}{n!},\end{align*} where $\epsilon_{k}=O(k^{-1})$.
Following Donaldson and Fine \cite{Do2}, \cite{Fi1} : Let $\phi \in L^2(X,End(E)),$ define $$(A_{\phi})_{ij}=\frac{N_{k}}{rV}\int_{X} \inner{s_{i}, \phi s_{j}}_{h\otimes \sigma^k} \frac{\omega^n}{n!}.$$
Note that $A_{\phi}$ defines a linear map from $H^0(E(k))$ into itself.
 i.e. we define $A_{\phi}s_{i}=\sum (A_{\phi})_{ij}s_{j}$ and extend it linearly to $H^0(E(k))$. Note that $A_{\phi}=\pi \circ M_{\phi} \circ j$, where $j: H^0(E(k)) \to L^2(X,E(k))$ is the inclusion, $\pi: L^2(X,E(k))\to H^0(E(k)) $ is the orthogonal projection and $M_{\phi}:  L^2(X,E(k)) \to L^2(X,E(k))$ is defined by $M_{\phi} s=\phi s.$ Hence, $$\norm{A_{\phi}}_{op}=\norm{\pi \circ M_{\phi}\circ j}_{op}\leq \norm{M_{\phi}}_{op}\leq \norm{\phi}_{C^0}.$$
Applying this to $\epsilon_{k}$, we have \begin{align*}\norm{\bar{\mu}(\Hilb(h))-\frac{rV}{N_{k}}Id}_{op}&=\norm{\int_{X}\inner{s_{i},\epsilon_{k}s_{j}}_{h\otimes \sigma^k} \frac{\omega^n}{n!}}_{op}=\frac{rV}{N_{k}}\norm{A_{\epsilon_{k}}}_{op}\\&\leq \frac{rV}{N_{k}}\norm{\epsilon_{k}}_{C^0}=O(k^{-n-1}). \end{align*}

\end{proof}

A consequence of Lemma \ref{0} is the following.

\begin{lem}\label{1}
 There is a constant $C>0$ such that for any matrices $A,B\in \sqrt{-1} \Lie{u}(N)$, one has
$$
\vert \Tr(AB\bar{\mu}) - \frac{1}{ k^n}\Tr(AB)\vert \leq Ck^{-n-1}\Tr(A^ 2)^{1/2}\Tr(B^2)^{1/2}.
$$
\end{lem}

\begin{proof}
Let $E:=\bar{\mu}-\frac{1}{ k^n}Id$. We have \begin{align*}  \vert \Tr(AB\bar{\mu}) - \frac{1}{ k^n}\Tr(AB)\vert&= \vert \Tr(ABE)\vert   \leq \norm{E}_{op} \vert \Tr(AB)\vert \\&\leq Ck^{-n-1}\norm{A}\norm{B}. \end{align*}

\end{proof}

\begin{lem}\label{2}
 There is a constant $C>0$ such that for any matrix $A\in \sqrt{-1} \Lie{u}(N)$, one has
$$\Vert H_A\Vert^2_{L^2} \leq \frac{1}{k^n}\left(1+ Ck^{-1}\right) \Tr(A^2)$$
\end{lem}

\begin{proof}
We have $$\int_{X}\Tr(H_A^2)\Omega+\norm{\xi_{A}}_{L^2}^2=\Tr(A^2\bar{\mu}).$$ Thus, 
$$\norm{H_{A}}_{L^2}^{2} \leq \Tr(A^2\bar{\mu}) \leq \frac{1}{k^n}\left(1+ Ck^{-1}\right) \Tr(A^2).$$

\end{proof}


\begin{lem}\label{5}
 Let $\psi \in L^2$ and let $M_k \in \sqrt{-1} \Lie{u}(N)$ be a sequence of $P_k^*P_k$-eigenvectors satisfying the following conditions
\begin{enumerate}
 \item $\Tr(M_k^2)=k^n + O( k^{n-1})$
\item $\Vert H_{M_k}- \psi\Vert^2_{L^2}=O(k^{-1})$
\end{enumerate}
then there is a constant $C>0$ such that for all $B \in \sqrt{-1} \Lie{u}(N_{k})$ with $\Tr(BM_k)=0$, we have
$$\vert \langle H_B,\psi\rangle_{L^2} \vert ^2 \leq Ck^ {-n-1}\Tr(B^2).$$
\end{lem}

\begin{proof}
We have $$\inner{H_B,H_{M_{k}}}_{L^2}=-\inner{\xi_{B},\xi_{M_{k}}}_{L^2}+ \Tr(BM_{k}\bar{\mu}).$$
On the other hand, \begin{align*}  \inner{\xi_{B},\xi_{M_{k}}}_{L^2}&=\inner{P_{k}B,P_{k}M_{k}}_{L^2}=\inner{B,P_{k}^*P_{k}M_{k}}_{L^2} \\&=\lambda \inner{B,M_{k}}=\lambda \Tr(BM_{k})=0.\end{align*}Moreover, lemma \ref{1} implies that $$\abs{\Tr(BM_{k}\bar{\mu})} \leq Ck^{-n-1}\norm{B}\norm{M_{k}}.$$
Hence, \begin{align*}  \abs{\inner{H_B,\psi}_{L^2}}&\leq \abs{\inner{H_B,H_{M_{k}}}_{L^2}}+\abs{\inner{H_B,H_{M_{k}}-\psi}_{L^2}}\\&\leq \abs{\inner{H_B,H_{M_{k}}}_{L^2}}+\norm{H_B}_{L^2}\norm{H_{M_{k}}-\psi}_{L^2}\\&=\abs{\Tr(BM_{k}\bar{\mu})}+\norm{H_B}_{L^2}\norm{H_{M_{k}}-\psi}_{L^2}\\&\leq Ck^{-\frac{n+1}{2}}\Tr(B^2)^{1/2}. \end{align*}

\end{proof}



\subsection{Lower bound for the eigenvalues}

The goal of this section is to prove the following lower bound for the eigenvalues which turns out to be much harder than the upper bound.
\begin{prop}\label{upperboundprop}
 Assume that $\lambda_r<\lambda_{r+1}$ and that the inductive hypothesis holds at level $r$. Then one has the following bound
 $$
 \nu_{r+1,k}\geq \frac{\lambda_{r+1}}{ 4 \pi k^{n+1}}+O(k^{-n-2}).
 $$
\end{prop}

The crucial step in the proof of Proposition \ref{upperboundprop} will be the following key-estimate

\begin{prop}\label{claim} 
For any $A\in \sqrt{-1}\mathfrak{u}(N_k)$, we have
 $$\Vert \nabla H_A\Vert^2_{L^2}\leq(4\pi k  + O(1))\norm{P_k(A)}^2.$$  
\end{prop}

The proof of this result makes use of the second fundamental form of a couple of holomorphic sub-bundles. In order to set things clear and for the sake of completeness, we begin by recalling some general theory about the second fundamental form. The way we present it here is close to Fine's treatment in \cite{Fi3}. Let $E\rightarrow X$ be a holomorphic Hermitian vector bundle over a complex manifold. Suppose $S$ is a holomorphic sub-bundle of $E$ with quotient $Q$. In other words, we have a short exact sequence
\begin{equation}
\label{ses}
0\rightarrow S \rightarrow E \rightarrow Q \rightarrow 0.
\end{equation}
Denote by $\nabla^E$ the Chern connection on $E$. By restriction we also get a Hermitian metric on $S$. Moreover, the Hermitian metric allows us to identify the quotient bundle $Q$ with $S^\perp$ as smooth vector bundles, so that we have a smooth splitting
$$
 E=S\oplus Q.
$$ 
Hence we also obtain a Hermitian metric on $Q$ which allows us to define Chern connections $\nabla^S$ and $\nabla^Q$ on $S$ and $Q$ respectively. It is easy to check that $\nabla^S$ is the composition of $\nabla^E$ followed by the projection to $S$. \\

There are two ways to look at the second fundamental form of a short exact sequence as in \eqref{ses}. Either you measure the failure of $S$ to be a parallel sub-bundle of $E$, or you look at $S^\perp$ and measure its  failure of being a holomorphic sub-bundle. The first point of view can be described as follows. Denote by $F$ the composition of $\nabla^E$ with the projection to $Q$. This defines an operator 
$$
F:C^\infty(S)\overset{\nabla^E}{\rightarrow}\Omega^1(E)\rightarrow\Omega^1(Q)
$$
called the \emph{second fundamental form} of \eqref{ses}. Note that since $S$ is a holomorphic sub-bundle, the $(0,1)$-part of $\nabla^E$ leaves $S$ invariant and thus $F$ is a section of the bundle $\Lambda^{1,0}\otimes\Hom(S,Q)$. \\

On the other hand, observe that if $S^\perp$ was a holomorphic sub-bundle, it would be invariant under $\delb^E$. The failure of $S^\perp$ of being a holomorphic sub-bundle can then be measured by the composition of $\delb^E$ with the projection to $S$. This defines a map
\begin{equation}
\label{dual_sec}
\tilde{F}:C^\infty(S^\perp)\overset{\delb^E}{\rightarrow}\Omega^{0,1}(E)\rightarrow\Omega^{0,1}(S).
\end{equation}
Hence we can think of $\tilde F$ as a section of $\Lambda^{0,1}\otimes\Hom(S^\perp,S)$.
One can check that under the identification $Q\simeq S^\perp$ the map $\tilde{F}$ is nothing else than $F^*$, the dual of $F$ obtained by using conjugation in the $(1,0)$-form factor and taking the usual adjoint in the $\Hom(S,Q)$ factor. \\

On one hand, write
$$
F\wedge F^*\in\Lambda^{1,1}\otimes\End(Q)
$$
where we take the genuine wedge product on the form part and composition on the homomorphism part. On the other hand, we consider
$$
F^*\wedge F\in\Lambda^{1,1}\otimes\End(S).
$$
Denote by $R(S)$, $R(Q)$ and $R(E)$ the curvatures of the Chern connections of $S$, $Q$ and $E$ respectively. By the splitting of $E=S\oplus Q$ as smooth vector bundles, we get an induced splitting
$$
\End(E)=\End(S)\oplus\Hom(S,Q)\oplus\Hom(Q,S)\oplus\End(Q).
$$
If we write now $R(E)|_S$ and $R(E)|_Q$ for the components of $R(E)$ in $\End(S)$ and $\End(Q)$ respectively, we have the following standard lemma. See for instance page 78 of \cite{griffithsharris} for a proof.

\begin{lem}
\label{2nd_fund_form_curvature}
\begin{align*}
F^*\wedge F&=R(S)-R(E)|_S \\
F\wedge F^*&=R(Q)-R(E)|_Q.
\end{align*}
\end{lem}

Assuming that the complex manifold $X$ carries a Hermitian metric, we can identify 
$$
\Lambda^{1,0}\simeq(\Lambda^{0,1})^*.
$$
Using this we can interpret $F$ as a homomorphism 
$$
F:\Lambda^{0,1}\otimes S\rightarrow Q
$$
and similarly $F^*$ as a homomorphism
$$
F^*:Q\rightarrow \Lambda^{0,1}\otimes S.
$$
The upshot is that these two maps are adjoint with respect to the fibrewise Hermitian metrics on $\Lambda^{0,1}\otimes S$ and $Q$. Furthermore we will be interested in the compositions $FF^*$ and $F^*F$ of these maps. Namely
$$
\Lambda^{0,1}\otimes S\overset{F}{\longrightarrow} Q \overset{F^*}{\longrightarrow} \Lambda^{0,1}\otimes S
$$
and
$$
Q\overset{F^*}{\longrightarrow}\Lambda^{0,1}\otimes S\overset{F}{\longrightarrow}Q.
$$
One can then check that under these identifications $F^*F$ is identified with $-F^*\wedge F$ whereas $FF^*$ is identified with $\Tr_X(F\wedge F^*)$. Here the trace is taken over the $\Lambda^{1,1}$-component of $F\wedge F^*$ using the Hermitian metric on $X$ (see Fine \cite{Fi3} page 28). \\

We will now use this theory in the following situation. Still suppose that we have a short exact sequence of holomorphic vector bundles
$$
0\rightarrow S \rightarrow E \rightarrow Q \rightarrow 0.
$$
Taking duals, we get another short exact sequence 
$$
0\rightarrow Q^* \rightarrow E^* \rightarrow S^* \rightarrow 0
$$
and taking the tensor product with the bundle $E$ yields
\begin{equation}
\label{ses_1}
0\rightarrow \Hom(Q,E) \rightarrow \End(E) \rightarrow \Hom(S,E) \rightarrow 0.
\end{equation}
The Hermitian metric on $E$ induces metrics on all of these bundles. Let $A\in\Gamma\left(\End(E)\right)$ be Hermitian and covariant constant with respect to the Chern connection on $\End(E)$, i.e.
$$
\nabla^{\End(E)}A=0.
$$
If we use the metric on $\End(E)$ to split 
$$
\End(E)=\Hom(Q,E)\oplus\Hom(S,E)
$$
as smooth vector bundles, we can write 
$$
A=
\begin{pmatrix}
A_1 \\
A_2
\end{pmatrix}
$$
where $A_1\in\Gamma(\Hom(Q,E))$ and $A_2\in\Gamma(\Hom(S,E))$. Furthermore we have 
$$
\delb^{\End(E)}=
\begin{pmatrix}
\delb^{\Hom(Q,E)} & \eta^* \\
0 & \delb^{\Hom(S,E)}
\end{pmatrix}
$$
where $\eta^*$ is the dual of the second fundamental form of the short exact sequence given in \eqref{ses_1}, defined as in \eqref{dual_sec}. Applying it to our covariant constant section $A$ yields
$$
\begin{pmatrix}
0 \\
0
\end{pmatrix}
=\delb^{\End(E)}A
=\begin{pmatrix}
\delb^{\Hom(Q,E)}A_1 + \eta^*A_2 \\
\delb^{\Hom(S,E)}A_2
\end{pmatrix}.
$$
In particular,
\begin{equation}
\label{a2_holomorphic}
\delb^{\Hom(S,E)}A_2=0
\end{equation}
meaning that $A_2\in\Gamma(\Hom(S,E))$ is a holomorphic section. \\

Now $\End(S)$ is a holomorphic sub-bundle of $\Hom(S,E)$ with quotient $\Hom(S,Q)$. In other words we have another short exact sequence
\begin{equation}
0\rightarrow\End(S)\rightarrow\Hom(S,E)\rightarrow\Hom(S,Q)\rightarrow 0.
\end{equation}
Again we use the Hermitian metric to split this sequence and write
$$
A_2=
\begin{pmatrix}
H_A \\
P_A
\end{pmatrix}
$$
where $H_A\in\Gamma(\End(S))$ and $P_A\in\Gamma(\Hom(S,Q))$. Writing
$$
\delb^{\Hom(S,E)}
=\begin{pmatrix}
\delb^{\End(S)} & F^*\\
0 & \delb^{\Hom(S,Q)}
\end{pmatrix}
$$
and applying it to the holomorphic section $A_2$ gives in particular
\begin{equation}
\label{relation_delbH_P}
\delb^{\End(S)}H_A=-F^* P_A.
\end{equation}
This formula is crucial for what follows. In fact it gives the geometric relation between the derivative of $H_A$ in terms of $P_A$. \\

In order to prove Proposition \ref{claim} we will now apply the above discussion to our picture. Recall that we used higher and higher powers of the line bundle $L$ tensored with $E$ to get a sequence of embeddings of $X$ into the Grassmannians $\Gr(r,N_k)$ which can be summarized by the following diagram,
$$
\begin{CD}
E(k)=\iota_k^*U_r @>>>U_r\\
@VVV@VVV\\
X @>\iota_k>>G(r,N_k)
\end{CD}
$$
We have the following short exact sequence of holomorphic vector bundles
$$
0\rightarrow U_r\rightarrow \underline{\C}^{N_k} \rightarrow Q\rightarrow 0 
$$
where $\underline{\C}^{N_k}$ denotes the trivial bundle over the Grassmannian. As explained, we can use the metric to identify the quotient $Q$ with $U_r^\perp$ as smooth vector bundles. 

Let $A\in\sqrt{-1}\Lie{u}(N_k)$. We may think of $A$ as a constant section of $\End(\underline{\C}^{N_k})$ so that we can apply the discussion from above. It is then just a matter of unwinding the definitions to see that the Hermitian endomorphism $H_A$ of $U_r$ defined in \eqref{H_A} coincides with the one described in the above discussion. Furthermore the holomorphic tangent bundle on the Grassmannian can be identified with $\End(U_r,Q)$. Under this identification, the section $P_k(A)$ of $T\Gr(r,N)_{|\iota_k(X)}$ defined by \eqref{P(A)} and Section \ref{Sect1} corresponds to the restriction to $\iota_k(X)$ of what we called $P_A$ just above. Formula \eqref{relation_delbH_P} gives then the link between the derivative of $H_A$ and $P_k(A)$ by
\begin{equation}
\label{crucial_rel}
\delb^{\End(U_r)}H_A=-F_k^* P_k(A).
\end{equation}
The next step in our discussion will be to control the asymptotics of the operator $F_kF_k^*$. However, it turns out to be easier to consider the operator $F_k^*F_k$ first and then pass to $F_k^*F_k$.

\begin{lem}
$\norm{F_k^*F_k-2\pi k\Id}_{C^0(\mathrm{op})}=O(1)$. Here $\Id$ denotes the identity in $\End\left(\Lambda^{0,1}\otimes\End\left(E(k)\right)\right)$ and $C^0(\mathrm{op})$ is the $C^0$-norm on sections of $\End\left(\Lambda^{0,1}\otimes\End\left(E(k)\right)\right)$ associated to the fibrewise operator norm.
\end{lem}

\begin{proof}
Recall that under the identification of $\Lambda^{1,0}$ with $(\Lambda^{0,1})^*$, $F_k^*F_k$ is identified with $-F_k^*\wedge F_k$. Moreover by lemma \ref{2nd_fund_form_curvature} we know that 
$$
-F_k^*\wedge F_k=R\left(\Hom(\iota_k^*U_r,\iota_k^*\underline{\C}^{N_k})\right)|_{\End(\iota_k^*U_r)}-R\left(\End(\iota_k^*U_r)\right)
$$
where the curvatures are computed with respect to the pull-back metrics. By definition of the embeddings, $E(k)=\iota_k^*U_r$, so that we can rewrite the right-hand side as
\begin{equation}
\label{diff_curv}
R\left(\Hom(E(k),\underline{\C}^{N_k})\right)|_{\End(E(k))}-R\left(\End\left(E(k)\right)\right).
\end{equation}
Let's start computing the first term. Since $\underline{\C}^{N_k}$ is flat, we get
\begin{align*}
R\left(\Hom(E(k),\underline{\C}^{N_k})\right)
&=R\left(E(k)^*\right)\otimes\Id_{\underline{\C}^{N_k}}, \\
&=-R\left(E(k)\right)^T\otimes\Id_{\underline{\C}^{N_k}}. 
\end{align*}
So we see that it boils down to calculate the curvature of $E(k)=E\otimes L^k$ computed with respect to the metric $\FS(h)\otimes \sigma^k$. Since on one hand $R(L^k)=-2\pi k\sqrt{-1}\omega$ and on the other hand $R(E)$ isn't growing in $k$, we get
$$
R(E(k))=O(1)+\Id_E\otimes R(L^k)=-2\pi k\sqrt{-1}\omega\otimes\Id_{E(k)}+O(1).
$$
Putting these together, we see that
$$
R\left(\Hom(E(k),\underline{\C}^{N_k+1})\right)|_{\End(E(k))}=2\pi k\sqrt{-1}\omega\otimes\Id_{\End(E(k))}+O(1).
$$
Furthermore, it is easy to see that $\End(E(k))=\End(E)$ so that the second term in \eqref{diff_curv} is of order $O(1)$. Hence, 
$$
-F_k^*\wedge F_k=2\pi k\sqrt{-1}\omega\otimes\Id_{\End(E(k))}+O(1).
$$
Raising indices to pass form $-F_k^*\wedge F_k$ to $F_k^*F_k$ proves the Lemma.
\end{proof}

We will now explain how to pass from $F_k^*F_k$ to $F_kF_k^*$. Denote by $T_k\in\End(\Hom(E(k),Q))$ the orthogonal projection onto the image of $F_k:\Lambda^{0,1}\otimes\End(E(k))\rightarrow \Hom(E(k),Q)$.

\begin{lem}
\label{asymptotics_FF*}
$\norm{F_kF_k^*-2\pi k T_k}_{C^0(\mathrm{op})}=O(1)$, where we use the $C^0$-norm on sections of $\End(\Hom(E(k),Q))$ associated to the fibrewise operator norm.
\end{lem}

\begin{proof}
The argument is essentially the same as the proof of Lemma $33$ in \cite{Fi3}, adapted to our situation. Accordingly, we give nearly word-by-word the same proof. Clearly we have that $\ker F_kF_k^*=\ker T_k$ and since $F_kF_k^*$ is self-adjoint, it is enough to prove that all the non-zero eigenvalues are given by $2\pi k+O(1)$.  But the non-zero eigenvalues of $F_kF_k^*$ and $F_k^*F_k$ are the same since the eigenvectors are matched up by $F_k^*$. The result then follows from the previous lemma.
\end{proof}

Having gathered all of these pre-requisites, we are finally in position to prove Proposition \ref{claim}. \\

\noindent\emph{Proof of Proposition \ref{claim}.} Let $A\in \sqrt{-1}\mathfrak{u}(N_k)$, we have
\begin{align*}
 \Vert \nabla^{\End(E)} H_A\Vert^2_{L^2}
 &= \int_X\tr\left(H_A\Delta^{\End(E)}H_A)\right)\Omega, \\
 &=\int_X\tr\left(H_A\left(2\Delta_{\delb} H_A-\sqrt{-1}[\Lambda F,H_A]\right)\right)\Omega, \\
 &=2\int_X \tr\left(H_A\Delta_{\delb} H_A\right) \Omega, \\
 &=2\int_X |\delb H_A|^2 \; \Omega.
 \end{align*}
Using the relation $\delb H_A=-F_k^* P_k(A)$ given in \eqref{crucial_rel} and lemma \ref{asymptotics_FF*} we further get that
\begin{align*}
\int_X |\delb H_A|^2 \; \Omega
&=\langle P_k(A), F_kF_k^* P_k(A) \rangle, \\
&=\langle P_k(A),\left(2\pi k T_k +O(1)\right) P_k(A) \rangle, \\
&\leq (2\pi k +O(1))\norm{P_k(A)}^2.
\end{align*}
This concludes the proof of Proposition \ref{claim}.
\qed
\vspace{0.5cm}

Assume that the induction hypothesis holds at level $r$ and let $\lambda_{r} < \lambda_{r+1}$. We have the following. 

\begin{lem}\label{3}
 Let $\phi_{0},\dots , \phi_{r}$ be an $L^2$-orthonormal basis for $E_{r}$ such that $\Delta^E \phi_{i}=\lambda_{i} \phi_{i}$. For integers $0<p<q\leq r$, satisfying  $\lambda_{p-1}<\lambda_p = \lambda_{p+1} = ... = \lambda_q <\lambda_{q+1}$ and $p\leq j\leq q$, let $A_{j,k} \in F_{p,q,r}$ are given by the induction hypotheses. Let $W_k\subset F_{r,k}$ the span of the vectors $A_{j,k}$ $(0\leq j \leq r).$ 
Then 
$$\nu_{r+1,k} \geq \min_{B \in W_k^\perp} \frac{\Vert P_k B\Vert^2}{\tr(B^2)}.$$
\end{lem}

 \begin{proof}
By hypothesis (2) of the induction (I), there is a constant $C$ such that $$\vert \Tr(A_{i,k}A_{j,k})- k^n\langle H_{A_{i,k}}H_{A_{j,k}} \rangle_{L^2_\Omega}\vert \leq Ck^{-1}(\Tr(A_{i,k}^2)\Tr(A_{j,k}^2))^{1/2}.$$ Therefore, $$\Tr(A_{i,k}^2)=k^{n}+O(k^{n-1})$$ and $$\Tr(A_{i,k}A_{j,k})=O(k^{n-1/2})\,\,\,\, \textrm{if}\,\,\,\,\, i\neq j,$$ since $H_{A_{i,k}}=\phi+O(k^{-1})$ uniformly by the induction hypotheses. Hence the vectors $A_{i,k}$ are linearly independent (otherwise their inner product would be of a similar order than their norms). Thus $\dim (W_k)=r+1$.  The minimal eigenvalue of $P_k^* P_k$ on $W_k^\perp$ is at least the $(r+2)^{th}$ eigenvalue $\nu_{r+1,k}$. Using the variational characterization of eigenvalues, one gets the required inequality.
 \end{proof}


\begin{prop}
There exists a constant $C$ such that $$\norm{P_{k}B}^2 \geq \Big(\frac{\lambda_{r+1}}{ 4\pi  k^{n+1}}+\frac{C}{k^{n+2}}    \Big) \Tr(B^2),$$  for all $B \in W_k^\perp.$
\end{prop}

\begin{proof}

\textbf{Step I:} We have $$\norm{H_B}_{L^2}^2+\norm{\xi_{B}}_{L^2}^2=\Tr(B^2\bar{\mu}).$$ This  together with Lemma \ref{1}, imply  that \begin{equation}\label{eq1R}\norm{H_B}_{L^2}^2+\norm{P_{k}B}_{L^2}^2\geq \frac{1}{k^n}(1+O(k^{-1}))\Tr(B^2),
\end{equation} 
since $\norm{P_{k}B}_{L^2}^2=\norm{\xi_{B}}_{L^2}^2.$

\textbf{Step II:} Let  $\phi_{0},\dots , \phi_{r}$ be an $L^2$-orthonormal basis for $E_{r}$ such that $\Delta^E \phi_{i}=\lambda_{i} \phi_{i}$ and let $$H_B= \sum_{j=0}^{r} \inner{H_B,\phi_{j}}_{L^2}\phi_{j}+\widetilde{H},$$ where $\widetilde{H}$ is orthogonal to $E_{r}.$ Applying Lemma \ref{5} to $A_{j,k},$ there exists a constant $C$ such that 
$$\abs{\inner{H_B, \phi_{j}}_{L^2}}^2 \leq Ck^{-n-1}\Tr(B^2),$$ for all $B \in W_k^\perp$ and $0\leq j \leq r.$ Therefore,
$$\norm{H_B}_{L^2}^2=\sum_{j=0}^{r}\abs{\inner{H_B, \phi_{j}}_{L^2}}^2 +\norm{\widetilde{H}}_{L^2}^2 \leq Ck^{-n-1}\Tr(B^2)+\norm{\widetilde{H}}_{L^2}^2,$$ for all $B \in W_k^\perp$. By definition, we have then
$$\lambda_{r+1} =\min_{\phi \in E_r^\perp} \frac{\Vert \nabla \phi\Vert^2_{L^2} }{\Vert \phi\Vert^2_{L^2}}\leq \frac{\Vert \nabla \widetilde{H}\Vert^2_{L^2} }{\Vert \widetilde{H}\Vert^2_{L^2}}.$$ Therefore,
$$\norm{\widetilde{H}}^2_{L^2} \leq \frac{1}{\lambda_{r+1}}\Vert \nabla \widetilde{H}\Vert^2_{L^2}.$$
On the other hand, \begin{align*}\Vert \nabla H_B\Vert^2_{L^2}=&\Vert \nabla \widetilde{H}\Vert^2_{L^2}+\Vert \nabla( H_B-\widetilde{H})\Vert^2_{L^2}\\
&+2  Re \inner{\nabla \widetilde{H},\nabla (H_B-\widetilde{H})}_{L^2}\cse\\=&\Vert \nabla \widetilde{H}\Vert^2_{L^2}+\Vert \nabla( H_B-\widetilde{H})\Vert^2_{L^2}.\end{align*} Actually we use here the fact that \begin{align*}
\inner{\nabla \widetilde{H},\nabla (H_B-\widetilde{H})}_{L^2}&=\inner{ \widetilde{H},\Delta^E (H_B-\widetilde{H})}_{L^2}
\\&= \sum_{j=0}^{r}\lambda_{j} \overline{\inner{H_B,\phi_{j}}_{L^2}}\inner{ \widetilde{H},\phi_{j}}_{L^2}\\&=0.
\end{align*}  

Hence, 
$$\norm{\widetilde{H}}^2_{L^2} \leq \frac{1}{\lambda_{r+1}}\Vert \nabla H_B\Vert^2_{L^2}.$$

\textbf{Step III:} Proposition \ref{claim}  implies that 

\begin{align*}  \norm{H_B}_{L^2}^2 &\leq Ck^{-n-1}\Tr(B^2)+\norm{\widetilde{H}}_{L^2}^2 \\&\leq \frac{1}{\lambda_{r+1}}\Vert \nabla H_B\Vert^2_{L^2}+Ck^{-n-1}\Tr(B^2) \\&\leq \frac{ 4\pi k }{\lambda_{r+1}}\norm{P_{k}B}^2+O(1)\norm{P_{k}B}^2+ Ck^{-n-1}\Tr(B^2).  
\end{align*}
 This together with \eqref{eq1R} conclude the proof.
\end{proof}

\begin{cor}\label{upperbound}
 Assume that $\lambda_r<\lambda_{r+1}$ and that the inductive hypothesis at level $r$ holds. Then one has the lower bound, 
 $$
 \nu_{r+1,k}\geq \frac{\lambda_{r+1}}{ 4 \pi k^{n+1}}+O(k^{-n-2}).
 $$
\end{cor}

\begin{proof}
We have from the previous Proposition $$\nu_{r+1,k} \geq \min_{B \in W_k^\perp} \frac{\Vert P_k B\Vert^2}{\tr(B^2)} \geq \frac{\lambda_{r+1}}{4 \pi k^{n+1}}+O(k^{-n-2}).$$

\end{proof}

\subsection{Completing the proof of the induction (I), steps 2 and 3}
In this subsection, we fix positive integers $r$ and $s$ such that $\lambda_r <\lambda_{r+1}= .. =\lambda_s <\lambda_{s+1}$. For any $A \in \sqrt{-1}\Lie{u}(N),$  we write \begin{equation}\label{split} {H_{A}}={H_{A}}^< + {H_{A}}^{r+1}+ {H_{A}}^>\end{equation}
where ${H_{A}}^<$ is the component of ${H_{A}}$ lying in $E_r$, ${H_{A}}^>$ lies in the span of the eigenspaces associated to eigenvalues strictly greater than $\lambda_{r+1}$ and ${H_{A}}^{r+1}$ is the component of ${H_{A}}$ in the span of the eigenspaces having eigenvalue $\lambda_{r+1}$.

\begin{prop}\label{prop1}
 If the $r^{\text{th}}$ inductive hypotheses hold, then there is a constant $C$ such that for all
 $A,B\in F_{s,k}$, 
$$\vert \Tr(AB)-  k^{n} \langle {H_{A}},H_B\rangle_{L^2}\vert \leq Ck^{-n-1}\tr(A^2)^{1/2}\tr(B^2)^{1/2}.$$
\end{prop}
\begin{proof}
 From Lemma \ref{1}, we know that there is a uniform constant $c>0$ such that
$$\vert \Tr(AB\bar{\mu})- \frac{1}{ k^n} \Tr(AB)\vert \leq ck^{-n-1} \Tr(A^2)^{1/2}\Tr(B^2)^{1/2}.$$
On the other hand, Lemma \ref{lemm0} implies that  $\Tr(AB\bar{\mu})=\Tr(AP_k^*P_k B)+\inner{ {H_{A}},H_B}_{L^2}$. Moreover, using the facts that $A$ and $B$ lie in $F_{s,k}$ and $\nu_{s,k}=O(k^{-n-1})$ we see that 
$$\abs{\Tr(AP_k^*P_k B)} \leq \frac{C}{k^{n+1}}\norm{A}\norm{B}.$$ Putting these estimates together concludes the proof.
\end{proof}
Next, we prove that the step 3 of the induction holds. We start with the following lemma.

\begin{lem}\label{6}
Assume that the $r^{\text{th}}$ inductive hypotheses hold. There exists $C$ such that for any $A \in F_{r+1,s,k},$ we have

\begin{align*}\Vert {H_{A}}^<\Vert_{L^2}^2 & \leq Ck^{-n-1}\Tr(A^2),\\ 
\Vert {H_{A}}^>\Vert_{L^2}^2 & \leq Ck^{-n-1}\Tr(A^2),\\
 \Big\vert k^n \Vert {H_{A}}^{r+1} \Vert_{L^2}^2 - \Tr(A^2) \Big\vert &\leq Ck^{-1} \Tr(A^2).
\end{align*}

\end{lem}

\begin{proof}
Without loss of generality, we may assume that $A \in F_{r+1,s,k}$ is a $\nu_{j,k}$ eigenvector of $P_k^*P_k$ with $r+1\leq j\leq s.$ Let $\phi_{0}, \dots, \phi_{r}$ be an orthonormal basis for $E_{r}$ such that $\Delta^E \phi_{j}=\lambda_{j}\phi_{j}$. 
Now, by the induction hypotheses, there are eigenvectors $A_{j,k}$ with eigenvalues $\nu_{j,k}\leq \nu_{r,k}$ for $P_k^*P_k$ satisfying $$\Tr(A_{j,k}^{ 2})=k^n +O(k^{n-1}),\,\,\,\, \,\,\, \Vert H_{A_{j,k}}-\phi_j\Vert_{L^2}=O(k^{-1/2}).$$ Since $A \perp A_{j,k}, 0\leq j\leq r,$ Lemma \ref{5} implies that $$\vert\langle {H_{A}},\phi_j\rangle_{L^2}\vert^2\leq Ck^{-n-1}\Tr(A^2).$$
Thus $\Vert {H_{A}}^<\Vert_{L^2}^2\leq Ck^{-n-1}\Tr(A^2).$\\
Moreover, Proposition \ref{prop1} implies that $$\Vert {H_{A}}^<\Vert^2_{L^2} + \Vert {H_{A}}^{r+1} \Vert^2_{L^2} + \Vert {H_{A}}^>\Vert^2_{L^2}=\frac{1}{k^n}(1+O(1/k))\tr(A^2)$$ and hence
\begin{align}\label{rel1b}
 \Vert {H_{A}}^{r+1}\Vert_{L^2}^2+ \Vert {H_{A}}^>\Vert_{L^2}^2 =\frac{1}{k^n}\left(1+O\left(\frac{1}{k}\right)\right) \Tr(A^2).
\end{align}
On the other hand, for any $A \in F_{r+1,s,k},$ an eigenvector associated to the eigenvalue $\nu_{j,k}$ ($r+1\leq j \leq s$), we have \begin{align}
\Vert \nabla {H_{A}}\Vert^2_{L^2}=&4\pi(k+O(1))\norm{P_{k}A}^2=4\pi(k+O(1))\nu_{j,k}\Tr(A^2) \nonumber \\  
=& \frac{(\lambda_{j}+O(k^{-1}))}{k^{n}}\Tr(A^2),                                           \label{rel1a}
                                          \end{align}
 since $\displaystyle\nu_{j,k}=\frac{\lambda_{j}}{ 4 \pi k^{n+1}}+O(k^{-n-2}).$
Using the splitting \eqref{split} and the fact that ${H_A}^{r+1}$ lies in the $\lambda_{r+1}$ eigenspace, we obtain from \eqref{rel1a},
\begin{align}\Vert\nabla {H_A}^< \Vert^2_{L^2} + \lambda_{r+1}\Vert {H_A}^{r+1} \Vert^2_{L^2} + \Vert \nabla {H_A}^> \Vert^2_{L^2} =\frac{(\lambda_{r+1}+O(k^{-1}))}{ 4 \pi k^{n}}\Tr(A^2).\label{rel2a}
\end{align}

The variational property for eigenvalues of $\Delta^E$ implies that 
$\displaystyle \lambda_{s+1}=\min_{\phi \in E_s^\perp} \frac{\Vert \nabla \phi\Vert^2_{L^2} }{\Vert \phi\Vert^2_{L^2}}.$ Therefore, 
$$\lambda_{s+1}\leq  \frac{\Vert \nabla {H_{A}}^>\Vert^2_{L^2} }{\Vert {H_{A}}^>\Vert^2_{L^2}},$$ since ${H_{A}}^> \in E_{s}^{\perp}.$ 
Thus, using the fact that $\Vert \nabla {H_{A}}^<\Vert_{L^2}^2\leq \lambda_r \Vert {H_{A}}^< \Vert^2_{L^2} \leq Ck^{-n-1}\tr(A^2)$, we obtain thanks to \eqref{rel2a},
\begin{align} \label{rel2b}
\lambda_{r+1}\Vert {H_{A}}^{r+1}\Vert_{L^2}^2+ \lambda_{s+1}\Vert {H_{A}}^>\Vert_{L^2}^2 &\leq &\frac{1}{k^n}\left(\lambda_{r+1}+O\left(\frac{1}{k}\right)\right) \Tr(A^2).
\end{align}
Since $\lambda_{s+1}>\lambda_{r+1}$, the system \eqref{rel1b}, \eqref{rel2b} ensures the existence of a constant $C>0$ such that 
\begin{align*}\Vert {H_{A}}^>\Vert_{L^2}^2&\leq Ck^{-n-1}\Tr(A^2),
\\ \Big\vert  k^n \Vert {H_{A}}^{r+1} \Vert_{L^2}^2 - \Tr(A^2) \Big\vert &\leq Ck^{-1} \Tr(A^2).
\end{align*}
The Lemma is proved.
\end{proof}

With this last proposition below, we obtain the induction at step $r+1$.
\begin{prop}\label{step3}
 Assume that $\lambda_r <\lambda_{r+1}= .. =\lambda_s <\lambda_{s+1}$ and the $r^{\text{th}}$ inductive hypotheses hold. Given $\phi \in \Ker(\Delta^E-\lambda_{r+1}Id)$ an eigenvector, let $A_{\phi,k}$ be the point of $F_{r+1,s}$ for which $H_{A_{\phi,k}}$ is nearest to $\phi \in L^2$. Then, 
$$\Vert H_{A_{\phi,k}}-\phi\Vert^2_{L^2}=O(k^{-1})$$
and this estimate is uniform in $\phi$ if in addition we require $\Vert \phi \Vert_{L^2}=1$.
\end{prop}

\begin{proof}
 First we show that the linear map $$A \in F_{r+1,s,k}  \to {H_{A}}^{r+1} \in V_{r+1}$$ is an isomorphism for $k \gg 0$,where $V_{r+1}$ is the eigenspace of $\Delta^E$ associated to the eigenvalue $\lambda_{r+1}.$ Suppose that  $A \in F_{r+1,s,k}$ and ${H_{A}}^{r+1}=0$. Then applying Lemma \ref{6}, we have
 $$\Big\vert  \Tr(A^2) \Big\vert \leq Ck^{-1} \Tr(A^2).$$ This implies that $A=0$ if $k \gg 0$. Note that 
 $\dim F_{r+1,s,k} \geq s-r=\dim V_{r+1}.$ Therefore, the linear map is an isomorphism. 
 This implies that for any $\phi \in V_{r+1},$ there exists a unique $A_{\phi,k}$ such that ${H_{A_{\phi,k}}}^{r+1}=\phi.$ Applying Lemma \ref{6}, we have
 $$\Vert H_{A_{\phi,k}}-\phi\Vert^2_{L^2}=\Vert {H_{A_{\phi,k}}}^< + {H_{A_{\phi,k}}}^>\Vert^2_{L^2}=O(k^{-1}).$$

\end{proof}

\section{Applications and Generalizations}
\subsection{Laplacian for balanced metrics}  
Consider a convergent sequence of balanced metrics in the sense of Wang on $E$, a simple holomorphic vector bundle. Then we know that $E$ is Mumford stable and conversely if $E$ is Mumford stable we know the existence of balanced metrics from \cites{W1,W2}. By uniformity in the previous asymptotics results, we get

\begin{thm} \label{bal1thm}Assume $E$ is Mumford stable.  For all large $k$, write $h_k\in Met(E)$ the balanced metric at level $k$. We consider the operators $Q_{k,.}$ and $P_k$ with respect to the balanced metric $h_k$. Let $\tilde{h}_{HE}$ be the almost Hermitian-Einstein metric on $E$ satisfying
$$\frac{\sqrt{-1}}{2\pi} \Lambda_\omega F_{(E,\tilde{h}_{HE})} = \left(\mu(E) + \frac{\bar s}{2} - \frac{S(\omega)}{2}\right)Id_E$$
where $\bar s$ is the average scalar curvature $S(\omega)$ of $\omega$. Then for any $\phi\in C^{\infty}(X,End(E))$, Hermitian with respect to $\tilde{h}_{HE}$, one has
\begin{equation}\label{bal1}\tr(Q_{k,\phi}P_k^*P_kQ_{k,\phi})\rightarrow \frac{1}{4\pi k}\int_X\tr( \phi\Delta^{E,\tilde{h}_{HE}}\phi) \Omega
 \end{equation}
where the Laplacian $\Delta^{E,\tilde{h}_{HE}}$ is computed with respect to $\tilde{h}_{HE}$. The result still holds if $\phi$ varies in a bounded set of Hermitian endomorphisms in the $C^{\infty}$-topology. \\
Furthermore, one has convergence of the eigenvalues $\nu_{j,k}$ of the operator $P_k^*P_k$ towards the  eigenvalues of $\Delta^{E,\tilde{h}_{HE}}$ after renormalization, i.e
$$4\pi k^{n+1}\nu_{j,k}\rightarrow {\lambda_j}.$$
Fix an integer $r>0$. There is a constant $C>0$ such that for all $A,B\in F_{r,k}$,
$$\Big\vert \Tr(AB)-k^n \langle {H_{A}},H_B\rangle_{L^2_\Omega} \Big\vert \leq Ck^{-1}\Tr(A^2)^{1/2}\Tr(B^2)^{1/2}.$$
Moreover, let us fix integers $0<p<q$ such that $$\lambda_{p-1}<\lambda_p = \lambda_{p+1} = ... = \lambda_q <\lambda_{q+1}.$$  Given $\phi \in \Ker(\Delta^{E,\tilde{h}_{HE}}- \lambda_p Id)$, let $A_{\phi,k}$ denote the point $F_{p,q,k}$ with $H_{A_{\phi,k}}$ nearest to $\phi$ as measured in the $L^2$-norm. Then 
$$\Vert H_{A_{\phi,k}}-\phi\Vert^2_{L^2_\Omega}= O(k^{-1})$$
and this estimate is uniform in $\phi$ if we require that $\Vert \phi\Vert_{L^2_\Omega}=1$.
\end{thm}
\begin{proof}
The balanced metric $h_k$ is a fixed point of the map $FS_k\circ Hilb_{k}: Met(E)\rightarrow Met(E)$ at level $k$. If we consider the sequence $$H_{k,l}=Hilb_{l}(h_k)\in Met(H^0(X,E(l))),$$ then the diagonal sequence $H_{k,k}$ is formed by balanced metrics. Let's  apply Theorem \ref{thm1} to the metrics $h_k$. We need to consider the family of operators $Q_{l,\phi,h_k}\in Herm(H^0(X,E(l)))$, for $l$ large enough and of course for $Q_{k,\phi,h_k}=Q_{k,\phi}$
 associated the balanced metric $h_k$. Similarly we introduce the operators $P_{l,h_k}$ that specify to $P_{k}$ when $l=k$. By construction of the balanced metric (see \cite{W2}), $\phi$ is also Hermitian with respect to all the $h_k$.\\
Therefore we can apply our previous results, and one has convergence 
$$\frac{1}{4\pi l}\tr(Q_{l,\phi,h_k}P_k^*P_kQ_{l,\phi,h_k})\rightarrow \int_X \phi\Delta^{E,h_k}(\phi) \Omega, $$ when $l\rightarrow + \infty$. Now \eqref{bal1} comes from the convergence in the smooth topology of $h_k$ towards $\tilde{h}_{HE}$. The other results of the theorem are obtained in a similar way by considering a diagonal argument and the uniformity in the results of convergence of Theorems \ref{thm3}, \ref{thm4}.
\end{proof}

\begin{cor}
 Under the assumptions of the above theorem, consider $h_{HE}$ the Hermitian-Einstein metric on $E$. Then, for any $\phi\in C^{\infty}(X,End(E))$, Hermitian with respect to ${h}_{HE}$, one has
\begin{equation*}\tr(Q_{k,\phi}P_k^*P_kQ_{k,\phi})\rightarrow \frac{1}{4\pi k}\int_X \phi\Delta^{E,h_{HE}}(\phi) \tilde{\Omega}
 \end{equation*}
where the Laplacian $\Delta^{E,h_{HE}}$ is computed with respect to the Hermitian-Einstein metric on $E$ and $\tilde{\Omega}=e^{\theta}\Omega$ with $\theta$ solution of 
the equation $\Delta_\omega \theta=\frac{1}{2}(S(\omega)-\bar{s})$ and $\int_X \tilde{\Omega}=\int_X e^{\theta}\Omega.$ Similar results as in Theorem \ref{bal1thm} hold for the eigenspaces and eigenvalues of the Laplacian $\Delta^{E,h_{HE}}$.
\end{cor}

\begin{rem}
 As we stressed earlier, when the metric on $E$ is Hermitian-Einstein, the Bochner Laplacian is related to the Kodaira Laplacian by the formula  $\Delta^{E,h_{HE}}=2\Delta_{\partial}=2\Delta_{\bar\partial}$.
\end{rem}

\subsection{Laplacian for $\Omega$-balanced metrics}

Let us restrict to the case of $E$ is the trivial line bundle. The case of $\Omega$-balanced metrics in the sense of \cite{D3} fits in the framework of the Theorem \ref{bal1thm}, since we know by \cite{Ke1} that they converge towards the solution to the Calabi problem. This gives the following result  as a corollary of the previous sections. 
\begin{thm}\label{thmlinebal}
For $P_k$ defined as above with respect to the $\Omega$-balanced metric $h_k$ on $L^k$, we have the asymptotic result
\begin{equation}\tr(Q_{k,f}P_k^*P_kQ_{k,f})\rightarrow \frac{1}{4\pi k}\int_X f\Delta_\omega(f) \Omega
 \end{equation}
for any $f \in C^{\infty}(X,\mathbb{R})$ and where $\omega$ satisfies the Monge-Amp\`ere equation $\omega^n/n!=\Omega$. The result holds if $f$ varies in a bounded subset of $C^{\infty}(X,\mathbb{R})$ in $C^\infty$-topology.
Furthermore, one has the convergence of the eigenvalues $\nu_{j,k}$ of the operator $P_k^*P_k$ towards the  eigenvalues $\lambda_j$ of $\Delta_\omega$ after renormalization, i.e
$$4\pi k^{n+1}\nu_{j,k}\rightarrow {\lambda_j}$$
Fix an integer $r>0$. There is a constant $C>0$ such that for all $A,B\in F_r$,
$$\Big\vert \Tr(AB)-k^n \langle {H_{A}},H_B\rangle_{L^2_\Omega} \Big\vert \leq ck^{-1}\Tr(A^2)^{1/2}\Tr(B^2)^{1/2}.$$
Moreover, let us fix integers $0<p<q$ such that $$\lambda_{p-1}<\lambda_p = \lambda_{p+1} = ... = \lambda_q <\lambda_{q+1}.$$  Given an eigenvector $\phi \in \Ker(\Delta_\omega- \lambda_p Id)$, let $A_{\phi,k}$ denote the point $F_{p,q}$ with $H_{A_{\phi,k}}$ nearest to $\phi$ as measured in $L^2$-norm. Then 
$$\Vert H_{A_{\phi,k}}-\phi\Vert^2_{L^2_\Omega}= O(k^{-1})$$
and this estimate is uniform in $\phi$ if we require that $\Vert \phi\Vert_{L^2_\Omega}=1$.
\end{thm}

\section{Explicit calculations for $\C P^n$.}

We will now illustrate our results via direct computations in the special case when our manifold is the complex projective space polarized by the dual of the tautological line bundle. We even get stronger results than the ones which hold in full generality. In fact, in this case the map $H:\sqrt{-1}\mathfrak{u}(N+1)\rightarrow C^\infty(\C P^N,\R)$ sends the eigenspaces of $P_k^*P_k$ to those of the Laplacian. Furthermore, we show that $H$ is an isometry (at least up to a constant) and that the eigenvalues of $P_k^*P_k$ converge to those of $\Delta$. Note that the fact that $H$ is an isometry between the eigenspaces of $P_k^*P_k$ is only true asymptotically in general. \\

Consider the sequence of embeddings $\iota_k:\C P^n\rightarrow \C P^N$ defined using the standard basis of $H^0\left(\C P^n,O(k)\right)$. Here we denote
$$
N+1=h^0\left(\C P^n,O(k)\right)={n+k\choose n}.
$$
By Lemma \ref{lemm0} we know that for all $A,B\in\sqrt{-1}\mathfrak{u}(N+1)$ we have that
\begin{equation}
\label{example_eq1}
\tr\left(A P_k^*P_k(B)\right)=\tr(AB\bar\mu_k)-\int_{\C P^N}H_AH_B\frac{\omega_{FS}^n}{n!}
\end{equation}
where $\mu([Z_0:\dots:Z_N])=\left(\frac{Z_i\bar Z_j}{\sum_\ell |Z_\ell|^2}\right)$ is the moment map of the unitary group acting on $\C P^N$ and $H_A=\tr(A\mu)$ is the Hamiltonian function associated to $A$. Note that if the embeddings $\iota_k$ are balanced, i.e. if
$$
(\bar{\mu}_k)_{ij}=\int_{\C P^N}\frac{Z_i\bar Z_j}{\sum_\ell |Z_\ell|^2}\frac{\omega_{FS}^N}{N!}=\frac{1}{N+1}\delta_{ij}
$$
then the first term of the right-hand side of \eqref{example_eq1} reduces to $\frac{1}{N+1}\tr(AB)$. \\

Let us start by recalling the spectral theorem for the Laplacian on complex projective space.

\begin{thm}[\cite{Gri}]
If $\Delta$ denotes the Laplacian on $\C P^n$ with respect to the Fubini-Study metric, one has the following:
\begin{enumerate}
\item The eigenvalues of $\Delta$ are given by $\lambda_i=4\pi i(i+n)$ 
 where $i\in\mathbb N$.
\item Denote by $\mathcal{W}_i=\{f\in C^\infty(\C P^n)\,|\,\Delta f=\lambda_i f\}$ the $i$-th eigenspace of $\Delta$. Then
$$
L^2(\C P^n)=\bigoplus_{i=0}^\infty \mathcal{W}_i.
$$
\item $U(n+1)$ acts on $\C P^n$ and induces an action on $\mathcal{W}_i$. Moreover $\mathcal{W}_i$ is an irreducible representation of $SU(n+1)$.
\item $\mathcal{W}_i$ consists of all functions of the form
$$
\frac{\sum_{|I|=|J|=i}\sqrt{{i\choose I}{i\choose J}}a_{IJ}Z^I\bar Z^J}{\left(\sum_\ell|Z_\ell|^2\right)^i}
$$
where $\sum_{|I|=|J|=i}\sqrt{{i\choose I}{i\choose J}}a_{IJ}Z^I\bar Z^J$ is a harmonic polynomial on $\C^{n+1}$. Here $I$ and $J$ denote multi-indices meaning that
\begin{align*}
I&=(i_0,\dots,i_n), \\
|I|&=i_0+\dots+i_n, \\
{i\choose I}&=\frac{k!}{i_0! i_1!\cdots i_n!},\\
Z^I&=Z_0^{i_0}\cdots Z_n^{i_n}, \\
a_{IJ}&\in\sqrt{-1}\mathfrak{u}(Sym^k\C^{n+1}).
\end{align*}
\end{enumerate}
\end{thm}

For the sake of clearness, we will restrict in the sequel to the case $n=1$, the sphere of volume $2\pi$. The general case is not fundamentally more difficult but the computations become a bit more messy. Furthermore we will consider the embedding of $\C P^1$ into $\C P^k$ that is balanced.\\

Fix $k\gg0$ and define 
$$
U_i=\{A\in\sqrt{-1}\mathfrak{u}(k+1) \,|\, H_A\in \mathcal{W}_i\}.
$$
Our goal is to show that the leading order of $P_k^*P_k$ restricted to $U_i$ is a multiple of the identity and that this multiple is precisely the $i$th eigenvalue of $\Delta$. This illustrates our general results from Theorems \ref{thm3} and \ref{thm4} that the eigenvalues of $P_k^*P_k$ converge to those of $\Delta$ and that eigenvectors converge isometrically under $H:\sqrt{-1}\mathfrak{u}(k+1)\rightarrow C^\infty(X,\R)$. \\

First observe that in our special case, $H:U_i\rightarrow \mathcal{W}_i$ is a $U(k+1)$-equivariant isomorphism. Since $\mathcal{W}_i$ is an irreducible real representation of $U(k+1)$ so is $U_i$. Furthermore it is easy to see that
$$
\inner{A,B}_1=\tr(AB)
$$
and
$$
 \inner{A,B}_2=\int_{\C P^1}H_A H_B \,\omega_{FS}
$$
define both $U(k+1)$-invariant inner products on $U_i$. This is trivial for the first one since the action on $\sqrt{-1}\mathfrak{u}(k+1)$ is given by conjugation. For the second one, observe that for $U\in U(k+1)$ we have
\begin{align*}
H_{U\cdot A}(z)&=H_{UAU^{-1}}(z)\\
&=\tr(AU^{-1}\mu_k(z)U)\\
&=\tr(A\mu_k(zU^{-1}))=H_A(zU^{-1}).
\end{align*}
Hence,
\begin{align*}
\inner{U\cdot A,U\cdot B}_2&=\int_{\C P^1}H_A(zU^{-1})H_B(zU^{-1}) \,\omega_{FS}\\
&=\int_{\C P^1}H_A(z) H_B(z) \,\omega_{FS}.
\end{align*}
Here the last equality follows from a change of variables and the fact that our embedding of $\C P^1$ into $\C P^k$ is supposed to be balanced. This implies that the volume form $\omega_{FS}$ is invariant. It follows then from a real version of Schur's lemma that both inner products only differ by a multiplicative constant. For the sake of completeness, let us briefly recall how this works. By the $U(k+1)$-invariance, there is an equivariant symmetric map $\varphi:U_i\rightarrow U_i$ such that 
$$
\inner{A,B}_2=\inner{\varphi(A),B}_1.
$$
On the other hand, since $\varphi$ is symmetric, $\varphi$ has a real eigenvalue and its associated eigenspace is invariant under $U(k+1)$.  The irreducibility of the representation then implies that the eigenspace is $U_i$. Therefore, $\varphi$ is a scalar matrix and thus the inner products differ by a real multiplicative constant. Hence there exists $C_{i,k}\in\R$ such that for any $A,B\in U_i$ we have
$$
\inner{A,B}_2=C_{i,k}\inner{A,B}_1.
$$
Furthermore in the balanced case, equation \eqref{example_eq1} implies that
\begin{equation}
\inner{A,P_k^*P_k(B)}_1=\left(\frac{1}{k+1}-C_{i,k}\right)\inner{A,B}_1
\end{equation}
which shows that $P_k^*P_k$ restricted to $U_i$ is a multiple of the identity. As a corollary, we get that $H$ sends eigenspaces of $P_k^*P_k$ isometrically to eigenspaces of $\Delta$, at least up to a constant. \\

The end of this section is devoted to compute the constants $C_{i,k}$. Clearly it is sufficient to find a particular $A\in U_i$ for which we can calculate both $\inner{\cdot,\cdot}_1$ and $\inner{\cdot,\cdot}_2$ explicitely. Their quotient gives then the required constant. \\

Denote by $[Z:W]$ homogeneous coordinates on $\C P^1$. The homogeneous polynomials 
$$
\sqrt{k \choose j}Z^j W^{k-j}
$$ 
for $j=0,\dots ,k$ define a basis of $H^0(\C P^1, O(k))$. Here the coefficients are chosen to make sure that the sections have unit norm. Furthermore, it turns out that the embedding they define into $\C P^{k}$ is balanced. \\

Consider the eigenfunction
$$
H_A=\frac{Z^i\bar W^i+\bar Z^i W^i}{\left(|Z|^2+|W|^2\right)^i}\in \mathcal{W}_i
$$
for some appropriate Hermitian matrix $A$. After multiplying the numerator and the denominator by $\left(|Z|^2+|W|^2\right)^{k-i}$ and developing that term we can write
\begin{align*}
H_A
&=\frac{\sum_{j=0}^{k-i}{k-i\choose j}|Z|^{2j}|W|^{2(k-i-j)}\left(Z^i\bar W^i+\bar Z^i W^i\right)}{\left(|Z|^2+|W|^2\right)^k} \\
&=2\frac{\sum_{j=0}^{k-i}   \sqrt{{k\choose i+j}{k\choose j}}^{-1}{k-i\choose j}  \textrm{Re}\left(\sqrt{{k\choose i+j}} Z^{i+j}W^{k-(i+j)} \sqrt{{k\choose j}} \bar Z^j\bar W^{k-j}\right)  }{\left(|Z|^2+|W|^2\right)^k}
\end{align*}
Therefore,
\begin{align*}
\tr(A^2)
&=2\sum_{j=0}^{k-i}\frac{{k-i\choose j}^2}{{k\choose i+j}{k\choose j}},\\
&=\frac{2\sum_{j=0}^{k-i}(k-j)\cdots(k-j-i+1)(j+2)\cdots(j+1) }{\left(k(k-1)\cdots(k-i+1)\right)^2}.
\end{align*}
A tedious but straightforward computation implies than that
\begin{align*}
\sum_{j=0}^{k-i}(k-j)&\cdots(k-j-i+1)(j+2)\cdots(j+1) \\
=&\;\frac{1}{(2i+1){2i\choose i}}\,k^{2i+1}+\frac{1}{{2i\choose i}}\,k^{2i}+O(k^{2i-1}).
\end{align*}
On the other hand we get
\begin{align*}
\left(k(k-1)\cdots(k-i+1)\right)^2
&=\left(k^i-\frac{i(i-1)}{2}k^{i-1}+O(k^{i-2})\right)^2 \\
&=k^{2i}-i(i-1)k^{2i-1}+O(k^{2i-1}).
\end{align*}
Putting these together, we get the following formula\footnote{Alternatively, one can prove by induction that $\tr(A^2)=\frac{ 4 (k+i+1)! (i+1)! (k-i) ! i! }{ (k!)^2 (2i+2)!}$.} 
\begin{align*}
\norm{A}_1^2=
\tr(A^2)
&=\alpha \,k\,\frac{1+(2i+1)k^{-1}+O(k^{-2})}{1-i(i-1)k^{-1}+O(k^{-2})}, \\
&=\alpha k \left(1+(i^2+i+1)k^{-1}+O(k^{-2})\right),
\end{align*}
where
$$
\alpha=\frac{2}{(2i+1){2i\choose i}}.
$$

To be able to deduce the constants $C_{i,k}$ we now only have to compute $\norm{A}_2^2$. 
\begin{align*}
\inner{A,A}_2
&=\int_{\C P^1}H_A^2\,\omega_{FS},\\
&=\int_{\C P^1}\frac{\left(Z^i\bar W^i + \bar Z^i W^i\right)^2}{\left(|Z|^2+|W|^2\right)^{2i}}\omega_{FS}, \\
&=2\textrm{Re}\int_{\C P^1}\frac{Z^{2i}\bar W^{2i}}{\left(|Z|^2+|W|^2\right)^{2i}}\omega_{FS}+2\int_{\C P^1}\frac{|Z|^{2i}|W|^{2i}}{\left(|Z|^2+|W|^2\right)^{2i}}\omega_{FS}.
\end{align*} 
By symmetry, the first of these integrals vanishes and the second one can be evaluated explicitly using  the local coordinate $z={W}/{Z}$. In fact we get
\begin{align*}
2\int_{\C P^1}\frac{|Z|^{2i}|W|^{2i}}{\left(|Z|^2+|W|^2\right)^{2i}}\omega_{FS}
&=\frac{\sqrt{-1}}{\pi}\int_\C \frac{|z|^{2i}}{(1+|z|^2)^{2i+2}} dz\wedge d\bar z \\
&=\frac{2\pi}{\pi(2i+1){2i\choose i}}=\alpha.
\end{align*}
Hence we proved the following Lemma.
\begin{lem}
For any $A,B\in U_i\subseteq\sqrt{-1}\mathfrak{u}(\textrm{Sym}^k\C^2)$ one has
$$
\int_{\C P^1}H_A H_B \,\omega_{FS}=\frac{1}{k}\left(1-(i^2+i+1)k^{-1}+O(k^{-2})\right)\tr(AB).
$$
\end{lem}
From here it is now easy to get the leading order term of the eigenvalues of $P_k^*P_k$ as expected from Theorem \ref{thm3}.

\begin{prop}
For any $A,B\in U_i\subseteq\sqrt{-1}\mathfrak{u}(\textrm{Sym}^k\C^2)$ one has
$$
\tr\left(A P_k^*P_k (B)\right)=\left(\frac{4\pi i(i+1)}{4\pi k^2}+O(k^{-3})\right)\tr(AB).
$$
\end{prop}
\begin{proof}
We have
\begin{align*}
\tr\left(A P_k^*P_k (B)\right)
&=\frac{1}{k+1}\tr(AB)-\int_{\C P^1} H_AH_B, \\
&=\left(\frac{1}{k+1}-\frac{1}{k}\left(1-(i^2+i+1)k^{-1}+O(k^{-2})\right)\right)\tr(AB), \\
&=\left(\frac{1}{k}-\frac{1}{k^2}-\frac{1}{k}+\frac{i^2+i+1}{k^2}+O(k^{-3})\right)\tr(AB), \\
&=\left(\frac{i(i+1)}{k^2}+O(k^{-3})\right) \tr(AB).
\end{align*}
\end{proof}



\begin{bibdiv}


\begin{biblist}

\bib{Bou}{article}{
    AUTHOR = {Bouche, Thierry},
     TITLE = {Convergence de la m\'etrique de {F}ubini-{S}tudy d'un fibr\'e
              lin\'eaire positif},
   JOURNAL = {Ann. Inst. Fourier (Grenoble)},
    VOLUME = {40},
      YEAR = {1990},
    NUMBER = {1},
     PAGES = {117--130},
      ISSN = {0373-0956},
}

\bib{C-K}{article}{
    AUTHOR = { Cao, H-D.},
    AUTHOR = { Keller, Julien},
     TITLE = {On the {C}alabi problem: a finite-dimensional approach},
   JOURNAL = {J. Eur. Math. Soc. (JEMS)},
    VOLUME = {15},
      YEAR = {2013},
    NUMBER = {3},
     PAGES = {1033--1065},
      ISSN = {1435-9855},
}

\bib{Ca}{article}{
    AUTHOR = {Catlin, David},
     TITLE = {The {B}ergman kernel and a theorem of {T}ian},
 BOOKTITLE = {Analysis and geometry in several complex variables ({K}atata,
              1997)},
    SERIES = {Trends Math.},
     PAGES = {1--23},
 PUBLISHER = {Birkh\"auser Boston, Boston, MA},
      YEAR = {1999},
}

\bib{Do2}{article}{
    AUTHOR = {Donaldson, S. K.},
     TITLE = {Scalar curvature and projective embeddings. {I}},
   JOURNAL = {J. Differential Geom.},
    VOLUME = {59},
      YEAR = {2001},
    NUMBER = {3},
     PAGES = {479--522},
      ISSN = {0022-040X},
}

\bib{D3}{article}{
    AUTHOR = {Donaldson, S. K.},
     TITLE = {Some numerical results in complex differential geometry},
   JOURNAL = {Pure Appl. Math. Q.},
    VOLUME = {5},
      YEAR = {2009},
    NUMBER = {2, Special Issue: In honor of Friedrich Hirzebruch. Part 
              1},
     PAGES = {571--618},
      ISSN = {1558-8599},
}


\bib{D4}{article}{
    AUTHOR = {Donaldson, S. K.},
     TITLE = {Remarks on gauge theory, complex geometry and {$4$}-manifold
              topology},
 BOOKTITLE = {Fields {M}edallists' lectures},
    SERIES = {World Sci. Ser. 20th Century Math.},
    VOLUME = {5},
     PAGES = {384--403},
 PUBLISHER = {World Sci. Publ., River Edge, NJ},
      YEAR = {1997}
}



\bib{Fi1}{article}{                                                                            
    AUTHOR = {Fine, Joel},                                                                      
     TITLE = {Calabi flow and projective embeddings},                                           
      NOTE = {With an appendix by Kefeng Liu and Xiaonan Ma},                                   
   JOURNAL = {J. Differential Geom.},                                                           
    VOLUME = {84},                                                                              
      YEAR = {2010},                                                                            
    NUMBER = {3},                                                                               
     PAGES = {489--523},                                                                        
      ISSN = {0022-040X},                                                                       
}

\bib{Fi3}{article}{                                                                            
    AUTHOR = {Fine, Joel},
     TITLE = {Quantization and the {H}essian of {M}abuchi energy},
   JOURNAL = {Duke Math. J.},
    VOLUME = {161},
      YEAR = {2012},
    NUMBER = {14},
     PAGES = {2753--2798},
      ISSN = {0012-7094},
 
}

\bib{Gri}{article}{
    AUTHOR = {Grinberg, Eric L.},
     TITLE = {Spherical harmonics and integral geometry on projective
              spaces},
   JOURNAL = {Trans. Amer. Math. Soc.},
    VOLUME = {279},
      YEAR = {1983},
    NUMBER = {1},
     PAGES = {187--203},
      ISSN = {0002-9947},
}

\bib{Ke1}{article}{
    AUTHOR = {Keller, Julien},
     TITLE = {Ricci iterations on {K}\"ahler classes},
   JOURNAL = {J. Inst. Math. Jussieu},
    VOLUME = {8},
      YEAR = {2009},
    NUMBER = {4},
     PAGES = {743--768},
} 

\bib{LM}{article}{
    AUTHOR = {Liu, Kefeng},
    AUTHOR = {Ma, Xiaonan},
     TITLE = {A remark on: ``{S}ome numerical results in complex
              differential geometry'' [arxiv.org/abs/math/0512625] by {S}.
              {K}. {D}onaldson},
   JOURNAL = {Math. Res. Lett.},
    VOLUME = {14},
      YEAR = {2007},
    NUMBER = {2},
     PAGES = {165--171},
      ISSN = {1073-2780},
}


\bib{Lu}{article}{
    AUTHOR = {Lu, Zhiqin},
     TITLE = {On the lower order terms of the asymptotic expansion of
              {T}ian-{Y}au-{Z}elditch},
   JOURNAL = {Amer. J. Math.},
    VOLUME = {122},
      YEAR = {2000},
    NUMBER = {2},
     PAGES = {235--273},
      ISSN = {0002-9327},
}


\bib{M-M2}{article}{
    AUTHOR = {Ma, Xiaonan},
    AUTHOR = {Marinescu, George},
     TITLE = {Berezin-{T}oeplitz quantization on {K}\"ahler manifolds},
   JOURNAL = {J. Reine Angew. Math.},
    VOLUME = {662},
      YEAR = {2012},
     PAGES = {1--56},
}

  

\bib{M-M}{book}{
    AUTHOR = {Ma, Xiaonan},
    AUTHOR = {Marinescu, George},
     TITLE = {Holomorphic {M}orse inequalities and {B}ergman kernels},
    SERIES = {Progress in Mathematics},
    VOLUME = {254},
 PUBLISHER = {Birkh\"auser Verlag, Basel},
      YEAR = {2007},
     PAGES = {xiv+422},
      ISBN = {978-3-7643-8096-0},
}

\bib{griffithsharris}{book}{
    AUTHOR = {Griffiths, Phillip},
    AUTHOR = {Harris, Joseph},
     TITLE = {Principles of algebraic geometry},
    SERIES = {Wiley Classics Library},
      NOTE = {Reprint of the 1978 original},
 PUBLISHER = {John Wiley \& Sons Inc.},
   ADDRESS = {New York},
      YEAR = {1994},
     PAGES = {xiv+813},
      ISBN = {0-471-05059-8},
}



\bib{W1}{article}{
   author={Wang, Xiaowei},
   title={Balance point and stability of vector bundles over a projective
   manifold},
   journal={Math. Res. Lett.},
   volume={9},
   date={2002},
   number={2-3},
   pages={393--411},
   issn={1073-2780},
   review={\MR{1909652 (2004f:32034)}},
}

\bib{W2}{article}{
   author={Wang, Xiaowei},
   title={Canonical metrics on stable vector bundles},
   journal={Comm. Anal. Geom.},
   volume={13},
   date={2005},
   number={2},
   pages={253--285},
   issn={1019-8385},
   review={\MR{2154820 (2006b:32031)}},
}


\bib{Ze}{article}{
   author={Zelditch, Steve},
   title={Szeg\H o kernels and a thm of Tian},
   journal={Internat. Math. Res. Notices},
   date={1998},
   number={6},
   pages={317--331},
   issn={1073-7928},
   review={\MR{1616718 (99g:32055)}},
   doi={10.1155/S107379289800021X},
}


\bib{Zh}{article}{
   author={Zhang, Shouwu},
   title={Heights and reductions of semi-stable varieties},
   journal={Compositio Math.},
   volume={104},
   date={1996},
   number={1},
   pages={77--105},
   issn={0010-437X},
   review={\MR{1420712 (97m:14027)}},
}


\end{biblist}
\end{bibdiv}
\end{document}